\documentclass[12pt, reqno, oneside]{amsart}

\usepackage{geometry}
\usepackage{amssymb}  
\usepackage[T1]{fontenc}
\usepackage{hyperref}  
\usepackage{mathtools}  
\usepackage{mathabx}  
\usepackage{arydshln}  

\usepackage{silence}
\usepackage[backend = biber, style = alphabetic, maxnames = 10, maxalphanames = 10, backref = false, backrefstyle = none, giveninits = true]{biblatex}
\renewbibmacro{in:}{%
    \ifentrytype{article}{}{\printtext{\bibstring{in}\intitlepunct}}}

\theoremstyle{definition}
\newtheorem{definition}{Definition}[section]
\newtheorem{example}[definition]{Example}

\theoremstyle{remark}
\newtheorem{remark}{Remark}[section]
\newtheorem*{notation*}{Notation}
\newtheorem*{observation*}{Observation}

\theoremstyle{plain}
\newtheorem{theorem}[definition]{Theorem}
\newtheorem{lemma}[definition]{Lemma}
\newtheorem{proposition}[definition]{Proposition}
\newtheorem{corollary}[definition]{Corollary}

\numberwithin{equation}{section}

\DeclareMathOperator{\GL}{GL}
\DeclareMathOperator{\Sp}{Sp}
\DeclareMathOperator{\SO}{SO}
\let\O\relax
\DeclareMathOperator{\O}{O}
\let\U\relax
\DeclareMathOperator{\U}{U}

\DeclareMathOperator{\V}{V}

\let\sp\relax
\DeclareMathOperator{\sp}{\mathrm{sp}}
\DeclareMathOperator{\so}{\mathrm{so}}
\let\o\relax
\DeclareMathOperator{\o}{\mathrm{o}}

\DeclareMathOperator{\jt}{JT}
\DeclareMathOperator{\sgn}{sgn}

\DeclareMathOperator{\GTP}{GTP}

\usepackage{letltxmacro}
\LetLtxMacro{\oldsum}{\sum}
\renewcommand{\sum}{\oldsum\limits}

\LetLtxMacro{\oldprod}{\prod}
\renewcommand{\prod}{\oldprod\limits}

\LetLtxMacro{\oldphi}{\phi}
\renewcommand{\phi}{\varphi}

\newcommand{\factorial}[3]{ ( #1 \mkern 2mu | \mkern 2mu #2 )^{ #3 } }

\DeclarePairedDelimiter{\abs}{\lvert}{\rvert}

\title[Ninth variation of classical group characters]{Ninth variation of classical group characters of type A--D and Littlewood identities}
\author{Mikhail Goltsblat}
\address{\parbox{\linewidth}{Department of Mathematics, Yale University, New Haven, CT 06511, USA\\
Faculty of Mathematics, HSE University, Moscow, Russia \textnormal{(prior to August 2022)}}}
\email{misha.goltsblat@yale.edu}

\begin{document}

\begin{abstract}
    We introduce certain generalisations of the characters of the classical Lie groups, extending the recently defined factorial characters of Foley and King. 
    In this extension, the factorial powers are replaced with an arbitrary sequence of polynomials, as in Sergeev-Veselov’s generalised Schur functions and Okada’s generalised Schur P- and Q-functions.
    We also offer a similar generalisation for the rational Schur functions.
    We derive Littlewood-type identities for our generalisations. 
    These identities allow us to give new (unflagged) Jacobi--Trudi identities for the Foley--King factorial characters and for rational versions of the factorial Schur functions.
    We also propose an extension of the original Macdonald's ninth variation of Schur functions to the case of symplectic and orthogonal characters, which helps us prove N\"{a}gelsbach--Kostka identities.
\end{abstract}

\maketitle


\section{Introduction}

\subsection{Motivation}
In the paper ``Schur functions: theme and variations'' I.~G.~Macdonald \cite{Macdonald} listed nine generalisations of Schur functions, which he called variations. The 6th variation is the factorial Schur functions. The ninth variation unifies variations~4--8.

The factorial Schur functions are defined as follows (see Macdonald~\cite{Macdonald}, Molev~\cite{Molev}). Take a sequence of variables $x_1, \dotsc, x_n$ and a doubly infinite sequence of indeterminates $c = (c_n)_{n \in \mathbb{Z}}$. Define the factorial (interpolation) powers 
\[
    \factorial{x}{c}{n} = (x - c_0) \dotsm (x - c_{n - 1}). 
\]
Given a partition $\lambda = (\lambda_1, \dotsc, \lambda_n)$ of length $\ell(\lambda) \le n$, the corresponding factorial Schur function $s_\lambda \factorial{x}{c}{}$ is the ratio of alternants
\[
    s_\lambda \factorial{x}{c}{} 
    =
    \dfrac
    {
        \det \left[ \factorial{x_i}{c}{\lambda_j + n - j} \right]_{i, j = 1}^n
    }
    {
        \det \left[ \factorial{x_i}{c}{n - j} \right]_{i, j = 1}^n
    }
    .
\]
It is easy to see that the denominator equals the usual Vandermonde determinant, so $s_\lambda \factorial{x}{c}{}$ is a (non-homogeneous) symmetric polynomial. Setting all $c_n = 0$, the polynomial $s_\lambda \factorial{x}{c}{}$ turns into the usual Schur polynomial $s_\lambda (x)$. In the $c_k = k$ specialisation, one recovers the shifted Schur polynomials $s_\lambda^*(y_1, \dotsc, y_n)$ in the variables $y_i = x_i - n + i$, see Okounkov--Olshanski~\cite{shifted}. The $c_k = q^{n - k - 1}$ specialisation gives the $q = t$ case of Olshanski's interpolation polynomials $I_\lambda(x_1, \dotsc, x_n; q, t)$, see Olshanski~\cite[(4.1)]{Olsh_2019}.

One can also define factorial analogues of the complete and elementary symmetric polynomials:
\begin{align*}
    & h_k \factorial{x}{c}{} = s_{(k, 0, \dotsc, 0)} \factorial{x}{c}{}, \qquad k \ge 0, \\
    & e_k \factorial{x}{c}{} = s_{(1^k)} \factorial{x}{c}{}, \qquad 0 \le k \le n.
\end{align*}
It turns out that the factorial Schur functions satisfy analogues of the Jacobi--Trudi, N\"{a}gelsbach--Kostka (dual Jacobi--Trudi), Giambelli, Cauchy and dual Cauchy identities, as well as a combinatorial formula representing them as a sum over semistandard Young tableaux. To state these formulas, introduce the shift operators $\tau^r$ for each $r \in \mathbb{Z}$ acting on the sequence~$c$ as 
\[
    (\tau^r c)_n = c_{n + r}.
\]
The Jacobi--Trudi and N\"{a}gelsbach--Kostka formulas take the form 
\begin{equation}\label{eq: unflagged factorial identity}
    s_\lambda \factorial{x}{c}{}
    =
    \det \left[ h_{\lambda_i - i + j} \factorial{x}{\tau^{1 - j} c}{} \right]_{i, j = 1}^n
    =
    \det \left[ e_{\lambda'_i - i + j} \factorial{x}{\tau^{j - 1} c}{} \right]_{i, j = 1}^{\lambda_1},
\end{equation}
where $\lambda'$ is the transpose of $\lambda$, $h_k \factorial{x}{\cdot}{} = e_k \factorial{x}{\cdot}{} = 0$ for $k < 0$, and $e_k \factorial{x}{\cdot}{} = 0$ for $k > n$, see Macdonald~\cite[(6.7)]{Macdonald}.
The Giambelli identity remains valid in its original form:
\[
    s_\lambda \factorial{x}{c}{}
    = 
    \det \left[ s_{ (\alpha_i | \beta_j) } \factorial{x}{c}{} \right]_{i, j = 1}^d,
\]
where $\lambda = (\alpha_1, \dotsc, \alpha_d | \beta_1, \dotsc, \beta_d)$ in the Frobenius notation. The Cauchy identity, due to Molev~\cite[Theorem~3.1]{Molev} and Olshanski in the special case of shifted Schur functions (see also Olshanski~\cite[Proposition~4.8]{Olsh_2019}), becomes 
\[
    \sum_{\substack{\lambda \, \textup{partition} \\ \ell(\lambda) \le n}}
    s_\lambda \factorial{x}{c}{} \widehat{s}_\lambda \factorial{u}{c}{} 
    =
    \dfrac
    {
        1
    }
    {
        \prod_{i, j = 1}^{n} (1 - x_i u_j) 
    }
    ,
\]
where $u_1, \dotsc, u_n$ is another sequence of variables.
Here the dual object, $\widehat{s}_\lambda \factorial{u}{c}{} $, is not a symmetric polynomial anymore, but a symmetric power series in $u_1, \dotsc, u_n$.

Macdonald's ninth variation of Schur functions is defined in the following way. Let $h_{rs}$ ($r \ge 1$, $s \in \mathbb{Z}$) be indeterminates. Also, set $h_{0s} = 1$ and $h_{rs} = 0$ for $r < 0$ and all $s$. Define an authomorphism $\phi$ of the ring generated by the $h_{rs}$ by $\phi (h_{rs}) = h_{r, s + 1}$, so that we can write $h_{rs} = \phi^s h_r$, where $h_r = h_{r 0}$. Let $\lambda$ be a partition of length $\ell(\lambda) \le n$. The ninth variation is defined as 
\[
    s_\lambda = \det \left[ \phi^{1 - j} h_{\lambda_i - i + j} \right]_{i, j = 1}^n.
\]
This emulates the factorial Jacobi--Trudi identity: if we specialise $\phi^s h_r = h_r \factorial{x}{\tau^s c}{}$, we obtain the Jacobi--Trudi determinant for the factorial Schur functions.

It turns out that already in the generality of Macdonald's ninth variation, one can prove the N\"{a}gelsbach--Kostka and Giambelli identities: we have 
\[
    s_\lambda 
    = 
    \det \left[ \phi^{j - 1} e_{\lambda'_i - i + j} \right]_{i, j = 1}^{\lambda_1} 
    = 
    \det \left[ s_{ (\alpha_i | \beta_j) } \right]_{i, j = 1}^d,
\]
where $e_k = s_{(1^k)}$ for $k \ge 0$, and $e_k = 0$ for $k < 0$, see Macdonald~\cite[(9.6'), (9.7)]{Macdonald}.

Subsequent to Macdonald's paper, several other generalisations of the Schur functions appeared, some of which are also referred to as ninth variations.

Nakagawa--Noumi--Shirakawa--Yamada~\cite{4authors} start with a bialternant formula and deduce flagged\footnote{The adjective \emph{flagged} means that the number of variables in an entry of the corresponding determinantal identity depends on its position in the matrix, as e.g.~in~\eqref{flagged}. Otherwise the identity is \emph{unflagged}. For the usual Schur polynomials $s_\lambda(x_1, \dotsc, x_n)$, the unflagged and the flagged Jacobi--Trudi identities are
\begin{align*}
    s_\lambda(x_1, \dotsc, x_n) = \det \left[ h_{\lambda_i - i + j}(x_1, \dotsc, x_n) \right]_{i, j = 1}^n,
    \\
    s_\lambda(x_1, \dotsc, x_n) = \det \left[ h_{\lambda_i - i + j}(x_1, \dotsc, x_{n + 1 - j}) \right]_{i, j = 1}^n.
\end{align*}

The second identity is due to Nakagawa--Noumi--Shirakawa--Yamada~\cite{4authors}. In the case of the usual Schur polynomials, it implies the first one. However, in the general situation of~\cite{4authors} only flagged Jacobi--Trudi and N\"{a}gelsbach--Kostka hold. Surprisingly, the Giambelli formula remains unflagged.} versions of the Jacobi--Trudi and N\"{a}gelsbach--Kostka identities, an unflagged Giambelli identity, and a combinatorial formula. 
Namely, take an infinite matrix of variables $X = (x_{ij})_{1 \le i, j < \infty}$. Given a partition $\lambda$ of length $\ell(\lambda) \le n$, assign to it the increasing sequence of integers 
\[
    \lambda_n + 1 < \lambda_{n - 1} + 2 < \dotsb < \lambda_1 + n.
\]
Let $\xi_\lambda^{(n)}(X)$ denote the minor of $X$ corresponding to the rows $1, \dotsc, n$ and columns $\lambda_n + 1, \lambda_{n - 1} + 2, \dotsc, \lambda_1 + n$. Nakagawa--Noumi--Shirakawa--Yamada define their ninth variation as the ratio 
\[
    S_\lambda^{(n)}
    =
    \dfrac{
        \xi_\lambda^{(n)}(X)
    }{
        \xi_\varnothing^{(n)}(X)
    }.
\]
If we specialise $x_{ij} = \factorial{x_i}{c}{j}$, we get the definition of the factorial Schur functions. Note that the superscript $(n)$ indicates the number of rows used, i.e.~the number of variables $x_i$ in the specialisation to factorial Schur functions.
One can also consider
\begin{align*}
    & h_k^{(n)} = S_{(k, 0, \dotsc, 0)}^{(n)}, \quad k \ge 0, \\
    & e_k^{(n)} = S_{(1^k)}^{(n)}, \quad k = 0, 1, \dotsc, n,
\end{align*}
and we also define $h_k^{(n)} = e_k^{(n)} = 0$ for $k < 0$.
Then the following flagged Jacobi--Trudi and N\"{a}gelsbach--Kostka formulas hold,
\begin{equation}\label{flagged}
    S_\lambda^{(n)}
    =
    \det \left[ h_{\lambda_i - i + j}^{(n + 1 - j)} \right]_{i, j = 1}^n
    =
    \det \left[ e_{\lambda'_i - i + j}^{(n + j - 1)} \right]_{i, j = 1}^{\lambda_1},
\end{equation}
as well as the (unflagged) Giambelli formula
\[
    S_\lambda^{(n)}
    =
    \det \left[ h_{(\alpha_i | \beta_j)}^{(n)} \right]_{i, j = 1}^d.
\]

In~\cite{Foley-King-ninth}, Foley and King start with a combinatorial formula and prove outside decomposition determinantal identities (Hamel--Goulden determinants), which include the Jacobi--Trudi, N\"{a}gelsbach--Kostka, and Giambelli identities. Since the present paper does not deal with combinatorial formulas, we omit the details. See also~\cite{Foley-King-super}.

Sergeev and Veselov~\cite{SergeevVeselov} consider a ninth variation associated with a sequence of polynomials, which they call \emph{generalised Schur functions}. 
They take a sequence of polynomials $(f_n (x))_{n \ge 0}$ satisfying a certain recurrence relation, which is not important for the purposes of this introduction, and consider the ratio 
\[
    S_\lambda (x) 
    =
    \dfrac{
        \det \left[ f_{\lambda_j + n - j} (x_i) \right]_{i, j = 1}^n
    }{
        \det \left[ f_{n - j} (x_i) \right]_{i, j = 1}^n
    }.
\]
The recurrence relation allows them to produce versions of the Jacobi--Trudi and Giambelli identities.
Functions of this kind also appear in Itoh~\cite{Itoh} and Olshanski~\cite{Olsh_2019} under the name of Schur-type functions. The Sergeev--Veselov generalisation is the one we are most concerned with in this paper.

These functions also appear in Okada~\cite{Okada}, where Schur’s $P$- and $Q$-functions associated with an arbitrary sequence of polynomials are introduced, and no recurrence relations on the sequence are assumed. Along with other results, Okada proves a Cauchy identity for the generalised Schur $P$-functions.

All the aforementioned works take place in the type A. 
Recently Foley and King~\cite{2018a} introduced factorial analogues of the characters of the other classical Lie groups (see also Foley--King~\cite{2018b}, where they introduce factorial analogues of the type BCD Schur $Q$-functions). They define them deforming the respective Weyl character formulas: writing $\overline{x}_i$ for $x_i^{-1}$, 
\begin{align*}
    \sp_\lambda \factorial{x}{c}{} 
    &= 
    \dfrac{
        \det \left[ x_i \factorial{x_i}{c}{\lambda_j + n - j} - \overline{x}_i \factorial{\overline{x}_i}{c}{\lambda_j + n - j} \right]
        }{
            \det \left[ x_i \factorial{x_i}{c}{n - j} - \overline{x}_i \factorial{\overline{x}_i}{c}{n - j} \right]
        }, \\
    \so_\lambda \factorial{x}{c}{}
    &= 
    \dfrac{ 
        \det \left[ x_i^{1/2} \factorial{x_i}{c}{\lambda_j + n - j} - \overline{x}_i^{1/2} \factorial{\overline{x}_i}{c}{\lambda_j + n - j} \right] 
        }{
            \det \left[ x_i^{1/2} \factorial{x_i}{c}{n - j} - \overline{x}_i^{1/2} \factorial{\overline{x}_i}{c}{n - j} \right]
        }, \\
    \o_\lambda \factorial{x}{c}{} 
    &= 
    \dfrac{ 
        \eta \det \left[ \factorial{x_i}{c}{\lambda_j + n - j} + \factorial{\overline{x}_i}{c}{\lambda_j + n - j} \right] 
        }{ 
            \frac{1}{2} \det \left[ \factorial{x_i}{c}{n - j} + \factorial{\overline{x}_i}{c}{n - j} \right]
        }, 
    \quad \hbox{where} \quad \eta = 
    \begin{cases}
        \frac{1}{2} & \hbox{if $\lambda_n = 0$}, \\
        1 & \hbox{if $\lambda_n > 0$}.
    \end{cases}
\end{align*}
Here $\so_\lambda \factorial{x}{c}{}$ is a deformation of an $\SO(2n + 1)$ character, while in the type D we work with the full orthogonal group $\O(2n)$, hence the notation $\o_\lambda \factorial{x}{c}{}$.
Note that setting all $c_n = 0$ recovers the Weyl character formulas in types CBD.
Foley and King derive flagged Jacobi--Trudi identities and combinatorial formulas for their factorial characters.
Our primary aim is to construct unflagged versions of the Jacobi--Trudi identities in the factorial setting.

\subsection{Results}
The present work is based upon the observation that instead of factorial powers, one can take an arbitrary sequence of polynomials in Foley--King's definition. This offers a Sergeev--Veselov kind of ninth variation for the symplectic and orthogonal characters (Definition~\ref{Definition: ninth CBD}). The results of Nakagawa--Noumi--Shirakawa--Yamada immediately allow us to give flagged Jacobi--Trudi and N\"{a}gelsbach--Kostka formulas, as well as an unflagged Giambelli formula.

Furthermore, we extend both the factorial Schur functions and the Sergeev--Veselov generalised Schur functions to the case when $\lambda$ is not a partition but an arbitrary dominant weight of $\GL_n$, i.e.~a signature, see Section~\ref{Subsection: ninth variation type A}. The factorial case of the extension seems to be known (see Jing--Rozhkovskaya~\cite[(5.3) and the unnumbered formula below (5.3)]{JingRozhkovskaya}), while the general case is new.

Most importantly, it turns out that our ninth variation characters (in all types) satisfy analogues of the Littlewood identities.
Classically, as the Cauchy identity describes the Howe duality between $\GL_n$ and $\GL_m$, the Littlewood identities describe Howe dualities for the symplectic/orthogonal pairs, see Howe~\cite{HoweRemarks}. 
Ninth variations of these identities are given in Theorems~\ref{Theorem: Sp Littlewood},~\ref{Theorem: SO(2n + 1) Littlewood},~\ref{Theorem: SO(2n) Littlewood}. Their type~A case is Theorem~\ref{Theorem: GL Littlewood}. It corresponds to a Howe duality between $\U(p, q)$ and $\U(p) \times \U(q)$, see~\cite[Section~6]{Olsh_Formalism}.

Throughout the article we work with the types BCD first since the proofs are less technically involved.

The Littlewood identities allow us to derive new unflagged Jacobi--Trudi formulas for the factorial characters of Foley and King, as well as a Jacobi--Trudi formula for the signature factorial Schur functions, see Theorems~\ref{Theorem: Jacobi-Trudi in type C},~\ref{Theorem: Jacobi-Trudi in type B},~\ref{Theorem: Jacobi-Trudi in type D},~\ref{Theorem: Jacobi-Trudi in type A}. This is possible due to the principle that once we have a Cauchy-type identity, which in our case is a ninth variation Littlewood identity, it is natural to expect that an identity of the Jacobi--Trudi form could be deduced from it (examples include \cite[Section 7.16]{Stanley}, \cite[Section 6]{borodin2011boundary}, \cite[Section~5.4]{aggarwal2021free}). For example, in the case of symplectic characters, our Jacobi--Trudi identity looks as follows:
\[
    \sp_\lambda \factorial{x}{c}{} = 
    \dfrac{1}{2} 
    \det \left[ h_{\lambda_i - i + j} \factorial{x, \overline{x}}{\tau^{1 - n - j} c}{} + h_{\lambda_i - i - j + 2} \factorial{x, \overline{x}}{\tau^{-1 - n + j} c}{} \right]_{i, j = 1}^n,
\]
where we assume that $c_n = 0$ for $n < 0$. Note that this assumption is also actively used in Foley--King~\cite{2018a}.

Our Jacobi--Trudi identities for the Foley--King factorial characters make it possible to define a ninth variation of symplectic and orthogonal characters in exactly the same way as Macdonald did originally. That is, we define 
\[
    \sp_\lambda 
    = 
    \dfrac{1}{2} 
    \det \left[ \phi^{1 - j} h_{\lambda_i - i + j} + \phi^{j - 1} h_{\lambda_i - i - j + 2} \right]_{i, j = 1}^n,
\]
and similarly for the orthogonal characters.
In this generality we prove symplectic and orthogonal N\"{a}gelsbach--Kostka identities~\eqref{eq: ninth sp nagelsbach-kostka},~\eqref{eq: ninth o nagelsbach-kostka}.
For example, 
\[
    \sp_\lambda 
    = 
    \det \left[ \phi^{j - 1} e_{\lambda'_i - i + j} - \phi^{j - 1} e_{\lambda'_i - i - j} \right]_{i,j = 1}^{\lambda_1}.
\]

Recall that the classical Schur polynomials can be lifted to symmetric functions (Macdonald~\cite[Chapters I.2--I.3]{Macdonald-book}). This is achieved by the fact that $s_\lambda(x_1, \dotsc, x_n, 0) = s_\lambda(x_1, \dotsc, x_n)$. Symplectic and orthogonal characters do not have such stability. Instead (say~in the symplectic case), Koike and Terada~\cite{KoikeTerada} propose to consider the symmetric functions
\[
    \dfrac{1}{2} 
    \det \left[ h_{\lambda_i - i + j} + h_{\lambda_i - i - j + 2} \right]_{i, j = 1}^n,
\]
which they call \emph{universal characters}. Here $h_m$ are the complete symmetric functions. Under the $h_m \mapsto h_m(x_1, \dotsc, x_n, x_n^{-1}, \dotsc, x_1^{-1})$ specialisation, the universal characters turn into the symplectic characters, hence providing a lift to the ring of symmetric functions. 
Macdonald's ninth variation plays the same role for the factorial characters. This way, the ninth variation $\sp_\lambda$ (as well as its orthogonal counterpart) is now a symmetric function and not a Laurent polynomial.

At the same time, the factorial Schur functions have a stability property similar to classical Schur polynomials: $s_\lambda \factorial{x_1, \dotsc, x_n, c_0}{c}{} = s_\lambda \factorial{x_1, \dotsc, x_n}{\tau c}{}$. Appropriately shifting the sequence $c$, Molev~\cite{Molev} defines \emph{double Schur functions}, which is another lift of factorial Schur functions to symmetric functions.
Additional relevant references include Olshanski~\cite{Olsh_2019} and Rains~\cite{Rains}.

\subsection*{Acknowledgements}

I am deeply grateful to Grigori Olshanski for his constant support and guidance.
I am indebted to Vadim Gorin for suggesting numerous edits to the text.
I also thank Darij Grinberg for pointing out several typos. 
Finally, I am grateful to the anonymous referee for their helpful comments and suggestions.


\section{Ninth variation characters}

In this section we recall determinantal expressions of the characters of the classical Lie groups and propose definitions for their ninth variations.

\subsection{Characters of classical groups}

Let us choose the upper-triangular Borel subroups and the diagonal Cartan subgroups in each of the classical Lie groups $\GL(n, \mathbb{C})$, $\Sp(2n, \mathbb{C})$, $\SO(2n + 1, \mathbb{C})$, $\SO(2n, \mathbb{C})$. We will be considering characters as functions on the Cartan subgroups in the following coordinates:
\begin{align*}
    & \operatorname{diag} (x_1, \dotsc, x_n), &  & \text{type A}, \\
    & \operatorname{diag} (x_1, \dotsc, x_n, x_n^{-1}, \dotsc, x_1^{-1}), &  & \text{types C--D}, \\
    & \operatorname{diag} (x_1, \dotsc, x_n, 1, x_n^{-1}, \dotsc, x_1^{-1}), &  & \text{type B},
\end{align*}
where $x_1, \dotsc, x_n \in \mathbb{C}^*$.

The dominant weights of the groups $\GL(n, \mathbb{C})$, $\Sp(2n, \mathbb{C})$, $\SO(2n + 1, \mathbb{C})$, $\SO(2n, \mathbb{C})$ are the integer sequences $\lambda = (\lambda_1, \dotsc, \lambda_n) \in \mathbb{Z}^n$ such that 
\begin{align*}
    & \lambda_1 \ge \dotsb \ge \lambda_n, &  & \text{type A}, \\
    & \lambda_1 \ge \dotsb \ge \lambda_n \ge 0,&  & \text{types C--B}, \\
    & \lambda_1 \ge \dotsb \ge \lambda_{n-1} \ge |\lambda_n|, &  & \text{type D}.
\end{align*}

The dominant weights of type A are called \emph{signatures}, while the dominant weights of types C--B are \emph{partitions}. 

In the type D, if $\lambda_n \ne 0$, there are two dominant weights, $\lambda_+ = (\lambda_1, \dotsc, \lambda_{n - 1}, \lambda_n)$ and $\lambda_- = (\lambda_1, \dotsc, \lambda_{n - 1}, -\lambda_n)$, where $\lambda_n > 0$. 
When $\lambda_n = 0$, the corresponding representation is a representation of the full orthogonal group $\O(2n, \mathbb{C})$, while when $\lambda_n \ne 0$, only the direct sum of the representations corresponding to $\lambda_+$ and $\lambda_-$ is a representation of $\O(2n, \mathbb{C})$.

We denote the characters of the groups $\GL(n, \mathbb{C})$, $\Sp(2n, \mathbb{C})$, $\SO(2n + 1, \mathbb{C})$ by, respectively, 
\[
    s_\lambda (x_1, \dotsc, x_n), \quad \sp_\lambda (x_1, \dotsc, x_n), \quad \so_\lambda (x_1, \dotsc, x_n).
\]
The characters $s_\lambda (x_1, \dotsc, x_n)$ are the \emph{rational} (also called \emph{composite}) Schur functions.

In the case of $\SO(2n, \mathbb{C})$, it is more convinient to work with the character 
\[
    \o_\lambda (x_1, \dotsc, x_n)
    = 
    \begin{cases}
        \chi_{\lambda}^{\SO(2n, \mathbb{C})} & \mbox{if $\lambda_n = 0$}, \\
        \chi_{\lambda_+}^{\SO(2n, \mathbb{C})} + \chi_{\lambda_-}^{\SO(2n, \mathbb{C})} & \mbox{if $\lambda_n > 0$},
    \end{cases}
\]
where $\lambda_1 \ge \dotsb \ge \lambda_n \ge 0$.

\begin{notation*}
    Note that we only use the notation $\so_\lambda$ for the odd orthogonal groups, and $\o_\lambda$ for the even orthogonal groups.
\end{notation*}

As follows from Weyl's character formula, they can be expressed as quotients of alternants:
\begin{align}
    s_\lambda (x_1, \dotsc, x_n) 
    &= 
    \dfrac{ \det \left[ x_i^{\lambda_j + n - j} \right] }{\det \left[ x_i^{n - j} \right]}, \label{Eq: gl} \\
    \sp_\lambda (x_1, \dotsc, x_n) 
    &= 
    \dfrac{ \det \left[ x_i^{\lambda_j + n - j + 1} - \overline{x}_i^{\lambda_j + n - j + 1} \right] }{\det \left[ x_i^{n - j + 1} - \overline{x}_i^{n - j + 1} \right]}, \label{Eq: sp} \\
    \so_\lambda (x_1, \dotsc, x_n) 
    &= 
    \dfrac{ \det \left[ x_i^{\lambda_j + n - j + 1/2} - \overline{x}_i^{\lambda_j + n - j + 1/2} \right] }{\det \left[ x_i^{n - j + 1/2} - \overline{x}_i^{n - j + 1/2} \right]}, \label{Eq: so} \\
    \o_\lambda (x_1, \dotsc, x_n) 
    &= 
    \dfrac{ \eta \det \left[ x_i^{\lambda_j + n - j} + \overline{x}_i^{\lambda_j + n - j} \right] }{ \frac{1}{2} \det \left[ x_i^{n - j} + \overline{x}_i^{n - j} \right]}, 
    \quad \hbox{where} \quad \eta = 
    \begin{cases}
        \frac{1}{2} & \hbox{if $\lambda_n = 0$}, \\
        1 & \hbox{if $\lambda_n > 0$}.
    \end{cases} \label{Eq: o}
\end{align}
The determinants above are of order $n$, and $\overline{x}_i = x_i^{-1}$.

The Weyl denominator formula tells us that the denominators of~\eqref{Eq: gl}-\eqref{Eq: o} are 
\begin{align}
    \det \left[ x_i^{n - j} \right] 
    &= 
    \V(x_1, \dotsc, x_n), \label{Eq: gl denominator} \\
    \det \left[ x_i^{n - j + 1} - \overline{x}_i^{n - j + 1} \right] 
    &= 
    \prod_{i = 1}^n (x_i - \overline{x}_i) 
    \cdot 
    \V(x_1 + \overline{x}_1, \dotsc, x_n + \overline{x}_n), \label{Eq: sp denominator} \\
    \det \left[ x_i^{n - j + 1/2} - \overline{x}_i^{n - j + 1/2} \right] 
    &= 
    \prod_{i = 1}^n (x_i^{1/2} - \overline{x}_i^{1/2}) 
    \cdot 
    \V(x_1 + \overline{x}_1, \dotsc, x_n + \overline{x}_n), \label{Eq: so denominator} \\
    \frac{1}{2} \det \left[ x_i^{n - j} + \overline{x}_i^{n - j} \right] 
    &= 
    \V(x_1 + \overline{x}_1, \dotsc, x_n + \overline{x}_n), \label{Eq: o denominator}
\end{align}
where $\V(x_1, \dotsc, x_n) = \prod_{1 \le i < j \le n} (x_i - x_j)$ is the Vandermonde determinant.

\subsection{Generalised Schur functions}

Let us recall the definition of Sergeev--Veselov's generalised Schur functions in the generality of Okada~\cite{Okada}.

\begin{definition}
    Let $\mathcal{F} = ( f_n(x) )_{n \ge 0}$ be a sequence of polynomials over a field $K$ of characteristic 0. Following Okada~\cite{Okada}, we call the sequence \emph{admissible} if 
    \begin{enumerate}
        \item $\operatorname{deg} f_n(x) = n$;
        \item the polynomials are monic;
        \item $f_0(x) = 1$.
    \end{enumerate}
    Such a sequence $\mathcal{F}$ is called \emph{constant-term free} if for any $n > 0$ the polynomial $f_n(x)$ does not have a constant term.
\end{definition}

Let $\lambda$ be a partition of length $\ell(\lambda) \le n$, and $\mathcal{F}$ be an admissible sequence of polynomials. Take a sequence of indeterminates $x_1, \dotsc, x_n$. The \emph{generalised Schur function} corresponding to the partition $\lambda$ and the sequence $\mathcal{F}$ is defined by 
\begin{equation}
    s_\lambda^\mathcal{F} (x_1, \dotsc, x_n) 
    = 
    \dfrac
    {
        \det \left[ f_{\lambda_j + n - j}(x_i) \right]_{i, j = 1}^n
    }
    {
        \det \left[ f_{n - j}(x_i) \right]_{i, j = 1}^n
    }
    .
    \label{Eq: generalised Schur function}
\end{equation}
Since the polynomials $f_n$ are monic, the denominator is the usual Vandermonde determinant $\V(x_1, \dotsc, x_n)$. Hence $s_\lambda^\mathcal{F} (x_1, \dotsc, x_n)$ is a symmetric polynomial in the $x_i$'s.

\begin{example}
    Let $c = (c_n)_{n \in \mathbb{Z}}$ be a sequence of indeterminates. Taking 
    \[
        f_k(x) = \factorial{x}{c}{k} = (x - c_0) \dotsm (x - c_{k - 1}),
    \]
    we recover the factorial Schur functions $s_\lambda \factorial{ x }{ c }{}$.
\end{example}

From the general results of Nakagawa--Noumi--Shirakawa--Yamada~\cite{4authors} (see Introduction), we have the flagged Jacobi--Trudi identity
\begin{equation}\label{eq: flagged identity}
    s_\lambda^\mathcal{F} (x_1, \dotsc, x_n) 
    = 
    \det \left[ s_{(\lambda_i - i + j)}^\mathcal{F} (x_1, \dotsc, x_{n - j + 1}) \right]_{i, j = 1}^n,
\end{equation}
the flagged N\"{a}gelsbach--Kostka identity
\[
    s_\lambda^\mathcal{F} (x_1, \dotsc, x_n) 
    =
    \det \left[ s_{\left(1^{\lambda'_i - i + j}\right)}^\mathcal{F} (x_1, \dotsc, x_{n + j - 1}) \right]_{i, j = 1}^{\lambda_1},
\]
and the Giambelli identity 
\[
    s_\lambda^\mathcal{F} (x_1, \dotsc, x_n) 
    = 
    \det \left[ s_{ (\alpha_i | \beta_j) }^\mathcal{F} (x_1, \dotsc, x_n) \right]_{i, j = 1}^d,
\]
where $\lambda = (\alpha_1, \dotsc, \alpha_d | \beta_1, \dotsc, \beta_d)$ in the Frobenius notation.

\subsection{Ninth variation C--B--D characters} 

Our definition of ninth variation characters in types C--B--D is inspired by the work of Foley and King~\cite{2018a}.

\begin{definition}
    Let $\lambda = (\lambda_1, \dotsc, \lambda_n)$ be a partition of length $\ell(\lambda) \le n$, and let $\mathcal{F} = ( f_n(x) )_{n \ge 0}$ be an admissible sequence of polynomials. As before, we denote~${\overline{x}_i = x_i^{-1}}$. The \emph{ninth variation characters} are 
    \begin{align}
        \sp_\lambda^\mathcal{F} (x_1, \dotsc, x_n) 
        &= 
        \dfrac{ \det \left[ x_i f_{\lambda_j + n - j}(x_i) - \overline{x}_i f_{\lambda_j + n - j}(\overline{x}_i) \right] }{\det \left[ x_i f_{n - j}(x_i) - \overline{x}_i f_{n - j}(\overline{x}_i) \right]}, \label{Eq: sp 9th} \\
        \so_\lambda^\mathcal{F} (x_1, \dotsc, x_n) 
        &= 
        \dfrac{ \det \left[ x_i^{1/2} f_{\lambda_j + n - j}(x_i) - \overline{x}_i^{1/2} f_{\lambda_j + n - j}(\overline{x}_i) \right] }{\det \left[ x_i^{1/2} f_{n - j}(x_i) - \overline{x}_i^{1/2} f_{n - j}(\overline{x}_i) \right]}, \label{Eq: so 9th} \\
        \o_\lambda^\mathcal{F} (x_1, \dotsc, x_n) 
        &= 
        \dfrac{ \eta \det \left[ f_{\lambda_j + n - j}(x_i) + f_{\lambda_j + n - j}(\overline{x}_i) \right] }{ \frac{1}{2} \det \left[ f_{n - j}(x_i) + f_{n - j}(\overline{x}_i) \right]}, 
        \quad \hbox{where} \quad \eta = 
        \begin{cases}
            \frac{1}{2} & \hbox{if $\lambda_n = 0$}, \\
            1 & \hbox{if $\lambda_n > 0$}.
        \end{cases} \label{Eq: o 9th}
    \end{align}
    \label{Definition: ninth CBD}
\end{definition}

Since the polynomials $f_n(x)$ are monic, the denominators coincide with the Weyl denominators~\eqref{Eq: sp denominator}-\eqref{Eq: o denominator}.

\begin{example}
    If we set $f_k(x) = \factorial{x}{c}{k}$, we recover the factorial characters of Foley and King~\cite{2018a}. We denote them by $\sp_\lambda \factorial{x}{c}{}$, $\so_\lambda \factorial{x}{c}{}$ and $\o_\lambda \factorial{x}{c}{}$.
\end{example}

As in the case of generalised Schur functions, from~\cite{4authors} we have the flagged Jacobi--Trudi, the flagged N\"{a}gelsbach--Kostka, and the Giambelli identities 
\begin{gather*}
    \sp_\lambda^\mathcal{F} (x_1, \dotsc, x_n) 
    = 
    \det \left[ \sp_{(\lambda_i - i + j)}^\mathcal{F} (x_1, \dotsc, x_{n - j + 1}) \right]_{i, j = 1}^n 
    = \\
    = 
    \det \left[ \sp_{\left(1^{\lambda'_i - i + j}\right)}^\mathcal{F} (x_1, \dotsc, x_{n + j - 1}) \right]_{i, j = 1}^{\lambda_1} 
    = 
    \det \left[ \sp_{ (\alpha_i | \beta_j) }^\mathcal{F} (x_1, \dotsc, x_n) \right]_{i, j = 1}^d,
\end{gather*}
and similarly for $\so_\lambda^\mathcal{F} (x_1, \dotsc, x_n)$ and $\o_\lambda^\mathcal{F} (x_1, \dotsc, x_n)$.
In the factorial case, the flagged Jacobi--Trudi identities were proven by Foley and King~\cite{2018a} by means of certain recurrence relations.

It is also immediate to obtain dual Cauchy identities for the ninth variation characters.

\begin{proposition} \label{Prop: dual Cauchy}
    Let $\mathrm{g} \in \{ s, \sp, \so, \o \}$. Then we have
    \[
        \sum_{\lambda \subset (m^n)} 
        (-1)^{\abs{ \widetilde{\lambda} } } 
        \mathrm{g}_\lambda^\mathcal{F}(x_1, \dotsc, x_n) 
        \mathrm{g}_{\widetilde{\lambda}}^\mathcal{F}(y_1, \dotsc, y_m)
        =
        \begin{cases}
            \prod_{i = 1}^n \prod_{j = 1}^m (x_i - y_j) & \mbox{\text{if} $\mathrm{g} = s$}, \\
            \prod_{i = 1}^n \prod_{j = 1}^m (x_i + \overline{x}_i - y_j - \overline{y}_j) & \mbox{\text{if} $\mathrm{g} = \sp, \so, \o$},
        \end{cases}
    \]
    where $\widetilde{\lambda} = (n - \lambda'_m, n - \lambda'_{m - 1}, \dotsc, n - \lambda'_1)$.
\end{proposition}
\begin{proof}
    In the type A, consider the determinant 
    \[
        \operatorname{Denom}_{n + m}^\mathcal{F} (x_1, \dotsc, x_n, y_1, \dotsc, y_m) 
        = 
        \left[
        \begin{array}{cccc}
        f_{n + m - 1} (x_1) & f_{n + m - 2} (x_1) & \cdots & f_0(x_1) \\
        \vdots & \vdots & & \vdots \\
        f_{n + m - 1} (x_n) & f_{n + m - 2} (x_n) & \cdots & f_0(x_n) \\
        f_{n + m - 1} (y_1) & f_{n + m - 2} (y_1) & \cdots & f_0(y_1) \\
        \vdots & \vdots & & \vdots \\
        f_{n + m - 1} (y_m) & f_{n + m - 2} (y_m) & \cdots & f_0(y_m)
        \end{array}
        \right].
    \]
    This is the type A Weyl denominator corresponding to the sequence $\mathcal{F}$. Since the sequence is admissible, the determinant is the same as for $\mathcal{F} = (x^n)_{n \ge 0}$.
    Hence the ratio 
    \[
        \dfrac
        {
            \operatorname{Denom}_{n + m}^\mathcal{F} (x_1, \dotsc, x_n, y_1, \dotsc, y_m)
        }
        {
            \operatorname{Denom}_{n}^\mathcal{F} (x_1, \dotsc, x_n) 
            \operatorname{Denom}_{m}^\mathcal{F} (y_1, \dotsc, y_m)
        }
    \]
    is equal to $\prod_{i = 1}^n \prod_{j = 1}^m (x_i - y_j)$.
    On the other hand, Laplace expansion of the determinant $\operatorname{Denom}_{n + m}^\mathcal{F} (x_1, \dotsc, x_n, y_1, \dotsc, y_m)$ along the first $n$ and last $m$ rows tells us that the ratio is equal to
    \[
        \sum_{\lambda \subset (m^n)} 
        (-1)^{\abs{ \widetilde{\lambda} } } 
        s_\lambda^\mathcal{F}(x_1, \dotsc, x_n) 
        s_{\widetilde{\lambda}}^\mathcal{F}(y_1, \dotsc, y_m).
    \]
    This proves the type A case, and all the other cases are dealt with similarly.
\end{proof}

\begin{remark}
    There is also a dual Cauchy identity for the Nakagawa--Noumi--Shirakawa--Yamada version of Macdonald's ninth variation, due to Noumi, see~\cite[(4.10)]{Zeta}.
\end{remark}

\begin{remark}
    The very same argument allows one to obtain dual versions of the Littlewood identities, see Sundaram~\cite[Theorem~4.9]{Sundaram}, i.e.~for each $\mathrm{g} \in \{ \sp, \so, \o \}$ there is a closed formula for the sum
    \[
        \sum_{\lambda \subset (m^n)} 
        (-1)^{\abs{ \widetilde{\lambda} } } 
        \mathrm{g}_\lambda^\mathcal{F} (x_1, \dotsc, x_n) 
        s_{\widetilde{\lambda}}^\mathcal{F} (y_1, \dotsc, y_m).
    \]
\end{remark}

\begin{remark}
    Macdonald~\cite[(6.19)]{Macdonald} shows that the infinite-variable limit of the factorial Schur function $s_\lambda \factorial{x_1, \dotsc, x_n}{c}{}$ is the supersymmetric Schur function $s_\lambda(x / -c)$ defined by 
    \[
        s_\lambda(x / -c) 
        = 
        \sum_{\mu \subset \lambda}
        s_\mu(x) s_{\lambda'/\mu'} (-c) \in \Lambda(x) \otimes \Lambda(c),
    \]
    where $\Lambda(x)$ and $\Lambda(c)$ are the rings of symmetric functions in variables $x = (x_1, x_2, \dotsc)$ and $c = (c_0, c_1, c_2, \dotsc)$.

    The same argument as in~\cite[(6.19)]{Macdonald} shows that the limit of the symplectic factorial character $\sp_\lambda \factorial{x}{c}{}$ is 
    \[
        \sum_{\mu \subset \lambda}
        \sp_\mu(x) s_{\lambda'/\mu'} (-c) \in \Lambda(x) \otimes \Lambda(c),
    \]
    where $\sp_\mu(x)$ is the universal symplectic character (see~\cite{KoikeTerada}). This expression is the infinite-variable limit of the orthosymplectic character $\operatorname{sc}_\lambda$ introduced by Benkart--Shader--Ram~\cite[Section 4]{Benkart}, studied later in~\cite{Stokke1},~\cite{Stokke2},~\cite{Stokke3}. The orthosymplectic characters $\operatorname{sc}_\lambda$ describe irreducible submodules in small tensor powers of the tautological representation of the orthosymplectic Lie algebra.

    Similarly, the limit of the orthogonal factorial characters is 
    \[
        \sum_{\mu \subset \lambda}
        \o_\mu(x) s_{\lambda'/\mu'} (-c) \in \Lambda(x) \otimes \Lambda(c),
    \]
    where $\o_\mu(x)$ is the universal orthogonal character.
\end{remark}

\subsection{Ninth variation type A characters} \label{Subsection: ninth variation type A}

Here we define certain ninth variation analogues of the rational Schur functions. In full generality these analogues are not rational, so we call them simply \emph{ninth variation type A characters}.

Let us first define the factorial Schur functions corresponding to signatures, this will motivate the general definition. For this we need to define the factorial powers $\factorial{x}{c}{n}$ for negative $n$.
Recall that we have the doubly infinite sequence $c = (c_n)_{n \in \mathbb{Z}}$ and the shift operators $\tau^r$ for each $r \in \mathbb{Z}$ defined by 
\[
    (\tau^r c)_n = c_{n + r}.
\]
It is clear that 
\[
    \factorial{x}{c}{r + s} = \factorial{x}{c}{r} \factorial{x}{ \tau^r c}{s}
\]
for all $r, s \ge 0$. Consequently, for a partition $\lambda$ of length $\ell(\lambda) \le n$ we have 
\[
    s_{\lambda + (1^n)} \factorial{ x }{ c }{} 
    = 
    \prod_{i = 1}^n \factorial{ x_i }{ c }{ 1 } \cdot s_\lambda \factorial{ x }{ \tau c }{} 
    = 
    s_{(1^n)} \factorial{ x }{ c }{} \cdot s_\lambda \factorial{ x }{ \tau c }{}.
\]
In order to preserve this property we set 
\begin{equation}
    \factorial{x}{c}{-n} = \dfrac{ 1 }{ (x - c_{-1}) \dotsm (x - c_{-n}) }
    \label{Eq: negative factorials}
\end{equation}
for $n > 0$. Then, indeed, 
\[
    \factorial{x}{c}{r + s} = \factorial{x}{c}{r} \factorial{x}{ \tau^r c}{s}
\]
for all $r, s \in \mathbb{Z}$. Now, given a signature $\lambda = (\lambda_1, \dotsc, \lambda_n)$ of length $\ell(\lambda) \le n$, we define the corresponding \emph{factorial Schur function} as 
\[
    s_\lambda \factorial{ x }{ c }{} = \dfrac{ \det \left[ \factorial{ x_i }{ c }{ \lambda_j + n - j } \right]_{i, j = 1}^n }{ \det \left[ \factorial{ x_i }{ c }{ n - j } \right]_{i, j = 1}^n }.
\]
The identity 
\[
    s_{\lambda + (1^n)} \factorial{ x }{ c }{} 
    = 
    s_{(1^n)} \factorial{ x }{ c }{} \cdot s_\lambda \factorial{ x }{ \tau c }{}
\]
holds true for any signature $\lambda$.
We note that~\eqref{Eq: negative factorials} is not new: it has already appeared in, for example, Jing--Rozhkovskaya~\cite[(5.3) and the unnumbered formula below (5.3)]{JingRozhkovskaya}.

\begin{observation*}
    Notice that $\factorial{x}{c}{-n}$ is a power series in $x^{-1}$ of order $n$ (recall that the \emph{order} of a formal power series $\sum_{k = 0}^\infty a_k y^k$ is the smallest integer $n$ such that $a_n \ne 0$).
\end{observation*}

Let us now treat the general case. 
Let $\mathcal{F} = (f_n(x))_{n \ge 0}$ be an admissible sequence of polynomials. We extend it to a doubly infinite sequence $\mathcal{F} = (f_n(x))_{n \in \mathbb{Z}}$ requiring $f_{-n}(x)$ to be a formal power series in $x^{-1}$ of order $n$, where $n > 0$. That is, for all $n > 0$, we set 
\[
    f_{-n}(x) = \sum_{k = n}^{\infty} a_{n, k} x^{-k}, \quad \text{where} \quad a_{n, n} \ne 0.
\]

Given a signature $\lambda = (\lambda_1, \dotsc, \lambda_n)$ of length $\ell(\lambda) \le n$, we define the \emph{ninth variation type A character} associated with the sequence $\mathcal{F} = (f_n(x))_{n \in \mathbb{Z}}$ by
\begin{equation}
    s_\lambda^\mathcal{F} (x_1, \dotsc, x_n) 
    = 
    \dfrac{ \det\left[ f_{\lambda_j + n - j}(x_i) \right]_{i, j = 1}^n }{\det\left[ f_{n - j}(x_i) \right]_{i, j = 1}^n }.
    \label{Eq: 9th variation type A}
\end{equation}
As before, the denominator is the Vandermonde determinant.

\begin{remark}
    Just like when $\lambda$ is a partition, there is the following flagged Jacobi--Trudi formula for the ninth variation type A characters:
    \[
        s_\lambda^\mathcal{F} (x_1, \dotsc, x_n) 
    = 
    \det \left[ s_{(\lambda_i - i + j)}^\mathcal{F} (x_1, \dotsc, x_{n - j + 1}) \right]_{i, j = 1}^n,
    \]
    where $\lambda$ is an arbitrary signature. 
    This formula is not stated explicitly in~\cite{4authors}, but their arguments can be easily modified to cover signatures, as the authors remark at the bottom of p. 181 in~\cite{4authors}. An identity of this type appeared in Borodin--Olshanski~\cite[Proposition 6.2]{borodin2011boundary}.
\end{remark}

\begin{remark}
    Foley and King~\cite{2018a} give combinatorial formulas for the factorial characters $\sp_\lambda \factorial{x}{c}{}$, $\so_\lambda \factorial{x}{c}{}$, $\o_\lambda \factorial{x}{c}{}$. One can also prove the following combinatorial formula for the factorial Schur function $s_\lambda \factorial{x}{c}{}$ associated with a signature~$\lambda$:
    \[
        s_\lambda \factorial{x}{c}{} 
        = 
        \sum_{G \in \GTP(\lambda)}
        \prod_{i = 1}^n \prod_{j = 1}^i
        \factorial{x_i}{\tau^{i - j + 1 + G_{i - 1, j}} c}{G_{i, j} - G_{i - 1, j}},
    \]
    where $\GTP(\lambda)$ is the set of all integral Gelfand--Tsetlin patterns with top row $\lambda$, and the entries of a pattern are enumerated as 
    \begin{equation*}
        G = \left(
        \begin{array}{ccccccccccccc}
         G_{n1} && G_{n2} && \cdots && G_{n,n-1} && G_{nn} \\
        & \cdots && \cdots && \cdots && \cdots \\
        && G_{31} && G_{32} && G_{33} \\
        &&& G_{21} && G_{22} \\
        &&&& G_{11} \\
        \end{array}
        \right).
    \end{equation*}
    That is, the entries are (not necessarily nonnegative) integers satisfying the inequalities $ G_{i + 1, j} \ge G_{ij} \ge G_{i + 1, j + 1} $ for all $i, j = 1, \dotsc, n$, and in the top row we put $G_{nj} = \lambda_j$ for each $j = 1, \dotsc, n$.
\end{remark}


\section{Cauchy identity}

In this section we state a ninth variation of the classical Cauchy identity.

\subsection{Dual sequences}

Consider the following non-degenerate bilinear pairing between the algebras $K[x]$ and $K[[u]]$:
\[
    \langle x^n, u^m \rangle = \delta_{n m}.
\]

\begin{remark}
    Our pairing is slightly different to that of Okada~\cite{Okada}, since it gives us Cauchy-type determinants, while Okada is working with Pfaffians instead.
\end{remark}

\begin{lemma}
    Let $\mathcal{F} = (f_n(x))_{n \ge 0}$ be an admissible sequence of polynomials. Then 
    \begin{enumerate}
        \item If $\widehat{\mathcal{F}} = (\widehat f_n(u))_{n \ge 0}$ is a sequence of formal power series in $u$ such that $\widehat f_n(u)$ has order $n$ for all $n \ge 0$, then 
        \[
            \langle f_n(x), \widehat f_m(u) \rangle = \delta_{nm} \quad \textup{if and only if} \quad \sum_{n = 0}^\infty f_n(x) \widehat f_n(u) = \dfrac{1}{1 - xu}.
        \]

        \item There exists a unique sequence $\widehat{\mathcal{F}}$ satisfying the equivalent conditions of~(1). We call such a sequence $\widehat{\mathcal{F}}$ the \emph{dual} of $\mathcal{F}$.

        \item If $\widehat{\mathcal{F}}$ is the dual of $\mathcal{F}$, then $\mathcal{F}$ is constant-term free if and only if $\widehat{f}_0 (u) = 1$.
    \end{enumerate}
    Conversely, given a sequence of formal power series $\widehat{\mathcal{F}} = (\widehat f_n(u))_{n \ge 0}$ such that $\widehat f_n(u)$ has order $n$ for all $n \ge 0$, there exists a unique sequence of polynomials $(f_n(x))_{n \ge 0}$ such that $f_n(x)$ has degree $n$ for all $n \ge 0$ and $\langle f_n(x), \widehat f_m(u) \rangle = \delta_{nm}$.
    \label{Lemma: dual sequences}
\end{lemma}

\begin{proof}
    The proof is the same as in~\cite[Lemma 3.1]{Okada}. 
\end{proof}

\begin{example}
    The dual of the sequence $f_k(x) = \factorial{x}{c}{k}$ is 
    \[
        \widehat f_k(u) = \dfrac{u^k}{ \prod_{l = 0}^k (1 - u c_l) }.
    \]
    This can be deduced from the identity
    \[
        \sum_{m = 0}^\infty \dfrac{ \factorial{x}{c}{m} }{\factorial{u}{c}{m + 1}} 
        = 
        \dfrac{1}{u - x},
    \]
    the proof of which is the same as in Okounkov--Olshanski~\cite[Theorem 12.1, first proof]{shifted}.
\end{example}

\subsection{Dual Schur functions}

Let $\mathcal{F} = (f_n(x))_{n \ge 0}$ be an admissible sequence of polynomials, and let $\widehat{\mathcal{F}} = (\widehat f_n(u))_{n \ge 0}$ be its dual.

\begin{definition}
    Given a partition $\lambda$ of length $\ell(\lambda) \le n$, the \emph{dual Schur function} is defined by
    \begin{equation}
        \widehat{s}_\lambda^\mathcal{F} (u_1, \dotsc, u_n) 
        = 
        \dfrac{ \det\left[ \widehat{f}_{\lambda_j + n - j}(u_i) \right]_{i, j = 1}^n }{ \V(u_1, \dotsc, u_n) },
        \label{Eq: dual Schur function}
    \end{equation}
    where $\V(u_1, \dotsc, u_n) = \prod_{1 \le i < j \le n} (u_i - u_j)$ is the Vandermonde determinant.
\end{definition}

Note that we \emph{require} the determinant to be the Vandermonde determinant. Thus, generally, the dual Schur function $\widehat{s}_\lambda^\mathcal{F} (u_1, \dotsc, u_n)$ does not coincide with $s_\lambda^{\widehat{\mathcal{F}}} (u_1, \dotsc, u_n)$, which is the generalised Schur function associated with the (non-polynomial) sequence $\widehat{\mathcal{F}}$.

\begin{example}
    In the factorial case, the dual Schur function takes the form 
    \begin{equation}
        \widehat{s}_\lambda \factorial{u}{c}{} 
        = 
        \dfrac
        {
            \det \left[ \dfrac{u_i^{\lambda_j + n - j}}{ \prod_{l = 0}^{\lambda_j + n - j} (1 - u_i c_l) } \right]
        }
        {
            \V(u_1, \dotsc, u_n)
        }
        .
        \label{Eq: dual factorial Schur function}
    \end{equation}
\end{example}

\subsection{Cauchy identity}
The functions $s_\lambda^\mathcal{F}(x_1, \dotsc, x_n)$ satisfy an analogue of the Cauchy identity. Its proof is based on an argument due to Ishikawa and Wakayama (see~\cite[Section 8]{Szego}). The key fact is the following lemma.

\begin{lemma}
    Let $\mathbf{f}_m (x)$ and $\mathbf{g}_m (u)$ be two sequences of functions, with $m \ge 0$. Let $x_1, \dotsc, x_n$ and $u_1, \dotsc, u_n$ be indeterminates. Then the following identity holds:
    \begin{equation}
        \sum_{\substack{\lambda \, \textup{partition} \\ \ell(\lambda) \le n}} 
        \det \left[ \mathbf{f}_{\lambda_j + n - j}(x_i) \right] \cdot \det \left[ \mathbf{g}_{\lambda_j + n - j}(u_i) \right] 
        = 
        \det \left[ \mathbf{h}(i, j) \right],
        \label{Eq: Cauchy--Binet}
    \end{equation}
    where $\mathbf{h}(i, j) = \sum_{m = 0}^\infty \mathbf{f}_m(x_i) \mathbf{g}_m(u_j)$ and all the determinants are $n \times n$.
    \label{Cauchy--Binet}
\end{lemma}

\begin{proof}
    Consider the following two infinite $n \times \infty$ matrices:
    \begin{align*}
        A = 
        \left[\begin{array}{cccc}
        \cdots & \mathbf{f}_2(x_1) & \mathbf{f}_1(x_1) & \mathbf{f}_0(x_1) \\
         & \vdots & \vdots & \vdots \\
        \cdots & \mathbf{f}_2(x_n) & \mathbf{f}_1(x_n) & \mathbf{f}_0(x_n)
        \end{array}\right], \\
        B = 
        \left[\begin{array}{ccccc}
        \cdots & \mathbf{g}_2(u_1) & \mathbf{g}_1(u_1) & \mathbf{g}_0(u_1) \\
         & \vdots & \vdots & \vdots   \\
        \cdots & \mathbf{g}_2(u_n) & \mathbf{g}_1(u_n) & \mathbf{g}_0(u_n)
        \end{array}\right].
    \end{align*}
    We number their columns by nonnegative integers in reverse order: $\dotsc, 2, 1, 0$. Let us calculate the determinant of the $n \times n$ matrix $AB^t$.

    On the one hand, the $ij$-th entry of $AB^t$ is $\mathbf{h}(i, j) = \sum_{m = 0}^\infty \mathbf{f}_m(x_i) \mathbf{g}_m(u_j)$, so $\det AB^t$ is the right-hand side of~\eqref{Eq: Cauchy--Binet}.
    On the other hand, by the Cauchy--Binet identity, 
    \[
        \det A B^t = \sum_{+ \infty > l_1 > \dotsb > l_n \ge 0} \det A_{l_1, \dotsc, l_n} \det B_{l_1, \dotsc, l_n},
    \]
    where $A_{l_1, \dotsc, l_n}$ and $B_{l_1, \dotsc, l_n}$ are the submatrices of $A$ and $B$ formed by the columns $l_1, \dotsc, l_n$. Writing $ (l_1, \dotsc, l_n) = (\lambda_1 + n - 1, \dotsc, \lambda_n) $ for a partition $\lambda$, we arrive at the left-hand side of~\eqref{Eq: Cauchy--Binet}.
    This proves the lemma.
\end{proof}

Let us apply the lemma to obtain the Cauchy identity.

\begin{proposition}
    We have
    \[
        \sum_{\substack{\lambda \, \textup{partition} \\ \ell(\lambda) \le n}}
        s_\lambda^\mathcal{F} (x_1, \dotsc, x_n) 
        \widehat{s}_\lambda^\mathcal{F} (u_1, \dotsc, u_n) 
        =
        \dfrac
        {
            1
        }
        {
            \prod_{i, j = 1}^{n} (1 - x_i u_j) 
        }
        .
    \]
    This is an identity in the ring of formal power series in $u_1, \dotsc, u_n$ with coefficients in the ring of symmetric polynomials in $x_1, \dotsc, x_n$.
\end{proposition}

\begin{proof}
    Set $\mathbf{f}_m (x) = f_m(x)$ and $\mathbf{g}_m (x) = \widehat f_m(u)$. Then 
    \[
        \mathbf{h}(i, j) 
        = 
        \sum_{n = 0}^\infty f_n(x_i) \widehat f_n(u_j) = \dfrac{1}{1 - x_i u_j},
    \]
    and, by Lemma~\ref{Cauchy--Binet}, we have 
    \begin{equation}
        \sum_{\substack{\lambda \, \textup{partition} \\ \ell(\lambda) \le n}} 
        \det \left[ f_{\lambda_j + n - j}(x_i) \right] \cdot \det \left[ \widehat f_{\lambda_j + n - j}(u_i) \right] 
        = 
        \det \left[ \dfrac{1}{1 - x_i u_j} \right].
        \label{Eq: almost Cauchy for partitions}
    \end{equation}
    The determinant on the right is the Cauchy determinant, which is equal to 
    \[
        \det \left[ \dfrac{1}{1 - x_i u_j} \right] 
        = 
        \dfrac
        {
            \V(x_1, \dotsc, x_n) \V(u_1, \dotsc, u_n)
        }
        {
            \prod_{i, j = 1}^{n} (1 - x_i u_j) 
        }.
    \]
    Hence, dividing~\eqref{Eq: almost Cauchy for partitions} by the Vandermondes $\V(x_1, \dotsc, x_n) \V(u_1, \dotsc, u_n)$, we arrive at the desired identity.
\end{proof}


\section{Littlewood identities}

In this section we introduce Littlewood identities for the ninth variation characters. See Littlewood~\cite[p. 240]{Littlewood_book},~\cite{Littlewood_article}, Koike--Terada~\cite[Lemma 1.5.1]{KoikeTerada} or Sundaram~\cite[Theorem~4.8]{Sundaram} for the classical case.
The proofs are based on Lemma~\ref{Cauchy--Binet}.

\subsection{Types C, B, D}

Hereafter we assume that $\mathcal{F}$ is an admissible polynomial sequence and $\widehat{\mathcal{F}}$ is its dual.

We start with the symplectic case.

\begin{theorem} \label{Theorem: Sp Littlewood}
    The following identity holds, where $\overline{x}_i = x_i^{-1}$,
    \begin{equation}
        \sum_{\ell(\lambda) \le n} 
        \sp_\lambda^\mathcal{F} (x_1, \dotsc, x_n) 
        \, 
        \widehat{s}_\lambda^\mathcal{F} (u_1, \dotsc, u_n) 
        = 
        \dfrac
        {
            \prod_{1 \le i < j \le n} (1 - u_i u_j)
        }
        {
            \prod_{i, j = 1}^n (1 - x_i u_j) (1 - \overline{x}_i u_j)
        }
        ,
        \label{Eq: Sp Littlewood}
    \end{equation}
    as an equality of formal power series in $u_1, \dotsc, u_n$ with coefficients in the ring of symmetric Laurent polynomials in $x_1, \dotsc, x_n$. The symmetry is with respect to the hyperoctahedral group $S_n \ltimes (\mathbb{Z}/2\mathbb{Z})^n$, where the $i$-th copy of $\mathbb{Z}/2\mathbb{Z}$ acts by $x_i \leftrightarrow \overline{x}_i$.
\end{theorem}

\begin{proof}
    Expand the left-hand side as 
    \begin{equation}
        \begin{split}
            & \sum_{\ell(\lambda) \le n} 
            \sp_\lambda^\mathcal{F} (x_1, \dotsc, x_n) 
            \, 
            \widehat{s}_\lambda^\mathcal{F} (u_1, \dotsc, u_n) 
            = \\
            & = 
            \sum_{\ell(\lambda) \le n} 
            \dfrac
            {
                1
            }
            {
                \prod_{i = 1}^n (x_i - \overline{x}_i) 
                \cdot 
                \V(x_1 + \overline{x}_1, \dotsc, x_n + \overline{x}_n)
            }
            \cdot
            \dfrac{1}{\V(u_1, \dotsc, u_n)}
            \cdot \\
            & \cdot 
            \det [ x_i f_{\lambda_j + n - j}(x_i) - \overline{x}_i f_{\lambda_j + n - j}(\overline{x}_i) ] 
            \cdot 
            \det [ \widehat{f}_{\lambda_j + n - j} (u_i) ].
        \end{split}
        \label{Eq: sum for Sp}
    \end{equation}
    Let us take $\mathbf{f}_n (x) = x f_n (x) - \overline{x} f_n (\overline{x})$ and $\mathbf{g}_n (u) = \widehat{f}_n (u)$ in Lemma~\ref{Cauchy--Binet} and evaluate the determinant $\det [ \mathbf{h}(i, j) ]$. First, 
    \begin{align*}
        & \mathbf{h}(i, j) 
        = 
        \sum_{m = 0}^\infty 
        ( x_i f_m (x_i) - \overline{x}_i f_m (\overline{x}_i) ) \widehat{f}_m (u_j) 
        = \\ 
        & = 
        x_i \cdot \dfrac{1}{1 - x_i u_j} 
        - 
        \overline{x}_i \cdot \dfrac{1}{1 - \overline{x}_i u_j} 
        = 
        \dfrac{ x_i - \overline{x}_i }{ (1 - x_i u_j) (1 - \overline{x}_i u_j) }.
    \end{align*}
    Hence our determinant is equal to 
    \[
        \det [ \mathbf{h}(i, j) ] 
        = 
        \prod_{i = 1}^n (x_i - \overline{x}_i) 
        \cdot 
        \det \left[ \dfrac{ 1 }{ (1 - x_i u_j) (1 - \overline{x}_i u_j) } \right].
    \] 
    The latter determinant can be calculated by reduction to the Cauchy determinant:
    \begin{align*}
        & \det \left[ \dfrac{ 1 }{ (1 - x_i u_j) (1 - \overline{x}_i u_j) } \right] 
        = \dfrac{1}{u_1 \dotsm u_n} \det \left[ \dfrac{ 1 }{ (u_i + \overline{u}_i) - (x_j + \overline{x}_j) } \right] 
        = \\ 
        & = 
        (-1)^{\binom{n}{2}} u_1^{n - 1} \dotsm u_n^{n - 1} 
        \V(u_1 + \overline{u}_1, \dotsc, u_n + \overline{u}_n) 
        \dfrac
        {
            \V(x_1 + \overline{x}_1, \dotsc, x_n + \overline{x}_n) 
        }
        {
            \prod_{i, j = 1}^n (1 - x_i u_j) (1 - \overline{x}_i u_j)
        } 
        = \\
        & = 
        \V(x_1 + \overline{x}_1, \dotsc, x_n + \overline{x}_n) 
        \cdot 
        \V(u_1, \dotsc, u_n) 
        \cdot 
        \dfrac
        {
            \prod_{1 \le i < j \le n} (1 - u_i u_j)
        }
        {
            \prod_{i, j = 1}^n (1 - x_i u_j) (1 - \overline{x}_i u_j)
        }
        ,
    \end{align*}
    where the last equality follows from the fact that
    \[
        (-1)^{\binom{n}{2}} u_1^{n - 1} \dotsm u_n^{n - 1} 
        \V(u_1 + \overline{u}_1, \dotsc, u_n + \overline{u}_n)
        =
        \V(u_1, \dotsc, u_n) \prod_{1 \le i < j \le n} (1 - u_i u_j).
    \]
    Substituting this into $\det [ \mathbf{h}(i, j) ]$, and $\det [ \mathbf{h}(i, j) ]$ into~\eqref{Eq: sum for Sp}, we arrive at the desired identity.
\end{proof}

The odd orthogonal case is proven in a similar way.

\begin{theorem} \label{Theorem: SO(2n + 1) Littlewood}
    The following identity holds, where $\overline{x}_i = x_i^{-1}$ :
    \begin{equation}
        \sum_{\ell(\lambda) \le n} 
        \so_\lambda^\mathcal{F} (x_1, \dotsc, x_n) 
        \, 
        \widehat{s}_\lambda^\mathcal{F} (u_1, \dotsc, u_n) 
        = 
        \dfrac{1}{\prod_{j = 1}^n (1 - u_j)}
        \cdot 
        \dfrac
        {
            \prod_{1 \le i \le j \le n} (1 - u_i u_j)
        }
        {
            \prod_{i, j = 1}^n (1 - x_i u_j) (1 - \overline{x}_i u_j)
        }
        .
        \label{Eq: SO(2n + 1) Littlewood}
    \end{equation}
\end{theorem}

\begin{proof}
    Take $\mathbf{f}_n (x) = x^{1/2} f_n (x) - \overline{x}^{1/2} f_n (\overline{x})$ and $\mathbf{g}_n (u) = \widehat{f}_n (u)$. Then
    \begin{align*}
        & \mathbf{h}(i, j) 
        = 
        \sum_{m = 0}^\infty 
        ( x_i^{1/2} f_m (x_i) - \overline{x}_i^{1/2} f_m (\overline{x}_i) ) \widehat{f}_m (u_j) 
        = \\ 
        & = 
        x_i^{1/2} \cdot \dfrac{1}{1 - x_i u_j} 
        - 
        \overline{x}_i^{1/2} \cdot \dfrac{1}{1 - \overline{x}_i u_j} 
        = 
        \dfrac{ (1 + u_j) (x_i^{1/2} - \overline{x}_i^{1/2}) }{ (1 - x_i u_j) (1 - \overline{x}_i u_j) }.
    \end{align*}
    Therefore, we have
    \[
        \det [ \mathbf{h}(i, j) ] 
        = 
        \prod_{j = 1}^n (1 + u_j) 
        \cdot
        \prod_{i = 1}^n (x_i^{1/2} - \overline{x}_i^{1/2}) 
        \cdot 
        \det \left[ \dfrac{ 1 }{ (1 - x_i u_j) (1 - \overline{x}_i u_j) } \right].
    \] 
    Now it suffices to notice that 
    \[
        \prod_{j = 1}^n (1 + u_j) \prod_{1 \le i < j \le n} (1 - u_i u_j) 
        = 
        \dfrac{\prod_{1 \le i \le j \le n} (1 - u_i u_j)}{\prod_{j = 1}^n (1 - u_j)},
    \]
    and the theorem follows.
\end{proof}

In the even orthogonal case we have to add the assumption that the sequence $\mathcal{F}$ is constant-term free.

\begin{theorem} \label{Theorem: SO(2n) Littlewood}
    Assume additionally that $\mathcal{F}$ is constant-term free. Then the following identity holds, where $\overline{x}_i = x_i^{-1}$:
    \begin{equation}
        \sum_{\ell(\lambda) \le n} 
        \o_\lambda^\mathcal{F} (x_1, \dotsc, x_n) 
        \, 
        \widehat{s}_\lambda^\mathcal{F} (u_1, \dotsc, u_n) 
        = 
        \dfrac
        {
            \prod_{1 \le i \le j \le n} (1 - u_i u_j)
        }
        {
            \prod_{i, j = 1}^n (1 - x_i u_j) (1 - \overline{x}_i u_j)
        }
        .
        \label{Eq: SO(2n) Littlewood}
    \end{equation}
\end{theorem}

\begin{proof}
    Take $\mathbf{f}_0 (x) = 1$, $\mathbf{f}_n (x) = f_n (x) + f_n (\overline{x})$ for $n > 0$, and $\mathbf{g}_n (u) = \widehat{f}_n (u)$ for all~$n$. Then 
    \begin{align*}
        & \mathbf{h}(i, j) 
        = 
        \widehat{f}_0 (u_j) 
        + 
        \sum_{m = 1}^\infty 
        ( f_m (x_i) + f_m (\overline{x}_i) ) \widehat{f}_m (u_j) 
        = \\ 
        & = 
        \dfrac{1}{1 - x_i u_j} 
        + 
        \dfrac{1}{1 - \overline{x}_i u_j} 
        - \widehat{f}_0 (u_j) 
        = 
        \dfrac{ 1 - u_j^2 }{ (1 - x_i u_j) (1 - \overline{x}_i u_j) }.
    \end{align*}
    The last equality holds under the assumption that $\mathcal{F}$ is constant-term free, since being constant-term free is equivalent to the condition $\widehat{f}_0 = 1$ by Lemma~\ref{Lemma: dual sequences}.
    The theorem now follows immediately.
\end{proof}

\begin{corollary}
    In particular, we have Littlewood identites for the factorial characters $\sp_\lambda \factorial{x}{c}{}$, $\so_\lambda \factorial{x}{c}{}$ and $\o_\lambda \factorial{x}{c}{}$ of Foley and King.
\end{corollary}

\subsection{Classical type A Littlewood identity}

Let $\lambda = (\lambda_1, \dotsc, \lambda_n)$ be a signature of length $\ell(\lambda) \le n$. It is possible to find a pair of partitions $\mu$ and $\nu$ of lengths $\ell(\mu) \le p$ and $\ell(\nu) \le q$, with $n = p + q$, such that 
\[
    (\lambda_1, \dotsc, \lambda_n) = (\mu_1, \dotsc, \mu_p, -\nu_q, \dotsc, -\nu_1).
\]
In such a case we write $\lambda = (\mu, \nu)$.
We want to obtain an analogue of the following formula.

\begin{proposition}[see Koike~{\cite[Lemma~2.5]{Koike}} or Olshanski~{\cite[Section~8]{Szego}}]
    Fix two nonnegative integers $p, q \ge 0$, and take $n = p + q$. 

    Let $\mu$ range over the set of all partitions of length $\le p$, and $\nu$ range over the set of all partitions of length $\le q$; write $\lambda = (\mu, \nu)$. 
    Take three sets of variables $x = (x_1, \dotsc, x_n)$, $u = (u_1, \dotsc, u_p)$, $v = (v_1, \dotsc, v_q)$. 

    Then the following identity holds:
    \begin{equation}
        \begin{gathered}
            \sum_{\lambda = (\mu, \nu)} 
            s_\lambda (x_1, \dotsc, x_n) 
            s_\mu (u_1, \dotsc, u_p) 
            s_\nu (v_1, \dotsc, v_q) 
            = \\
            =
            \dfrac
            {
                \prod_{i = 1}^{p} \prod_{j = 1}^{q} (1 - u_i v_j)
            }
            {
                \prod_{k = 1}^{n} \prod_{i = 1}^{p} (1 - x_k u_i) 
                \cdot 
                \prod_{k = 1}^{n} \prod_{j = 1}^{q} (1 - x_k^{-1} v_j)
            }
            .
        \end{gathered}
        \label{Eq: classical GL Littlewood}
    \end{equation}
\end{proposition}

As we can see, there are two dual objects in this formula: $s_\mu (u_1, \dotsc, u_p)$ and $s_\nu (v_1, \dotsc, v_q)$. The first one will be replaced by the dual Schur function $\widehat{s}_\lambda^\mathcal{F} (u_1, \dotsc, u_p)$. There is a bit more work to be done to replace the second one.

\subsection{Double dual Schur functions}

Recall that we extended the polynomial sequence $\mathcal{F}$ to a doubly infinite sequence $\mathcal{F} = (f_n(x))_{n \in \mathbb{Z}}$ requiring that $f_{-n}(x)$ is a power series in $x^{-1}$ of order $n$.

Consider the sequence $\mathcal{F}_{\le 0} = \left( f_{-n}(x) \right)_{n \ge 0}$, which consists of formal power series in $x^{-1}$. By means of Lemma~\ref{Lemma: dual sequences}, construct a sequence of polynomials $\widecheck{\mathcal{F}} = \left( \widecheck{f}_n (v) \right)_{n \ge 0}$ in $v$ that is dual to the sequence $\mathcal{F}_{\le 0}$.
In other words, the sequence $\widecheck{\mathcal{F}}$ is defined by the condition that $\widecheck{f}_n (v)$ has degree $n$ and $ \sum_{n = 0}^\infty f_{-n}(x) \widecheck{f}_n (v) = \dfrac{1}{1 - x^{-1} v}.$
Such a sequence is unique but need not be monic. We call this sequence $\widecheck{\mathcal{F}}$ the \emph{double dual} of $\mathcal{F}$. Note also that since $f_0(x) = 1$, the sequence $\widecheck{\mathcal{F}}$ is constant-term free by the part (3) of Lemma~\ref{Lemma: dual sequences}.

\begin{example}\label{Example: double dual}
    Let us calculate the double dual of $\mathcal{F} = ( \factorial{x}{c}{n} )_{n \in \mathbb{Z}}$. Consider the reversed sequence $\widetilde{c}$ defined by $\widetilde{c}_n = c_{-n - 1}$ for all $n \in \mathbb{Z}$.
    We claim that $\widecheck{f}_n (v) = v \factorial{v}{\widetilde{c}}{n - 1}$ for $n > 0$, and $\widecheck{f}_0 (v) = 1$.

    \begin{proof}
    First, recall that $\factorial{x}{c}{-n} = \dfrac{ 1 }{ (x - c_{-1}) \dotsm (x - c_{-n}) }$ for $n > 0$.
    We can rewrite it as 
    \[
        \begin{split}
            \factorial{x}{c}{-n} = \dfrac{ x^{-n} }{ (1 - x^{-1} c_{-1}) \dotsm (1 - x^{-1} c_{-n}) } = \\
            = \dfrac{ x^{-n} }{ (1 - x^{-1} \cdot 0) (1 - x^{-1} c_{-1}) \dotsm (1 - x^{-1} c_{-n}) }.
        \end{split}
    \]
    Now, we know that
    \[
        \sum_{n = 0}^\infty \factorial{x}{c}{n} \dfrac{u^n}{\prod_{l = 0}^n (1 - u c_l)} 
        =
        \dfrac{1}{1 - xu}.
    \]
    From this equality we immediately deduce
    \[
        \sum_{n = 0}^\infty v \factorial{v}{\widetilde{c}}{n - 1} \dfrac{ x^{-n} }{ (1 - x^{-1} \cdot 0) (1 - x^{-1} c_{-1}) \dotsm (1 - x^{-1} c_{-n}) } 
        = 
        \dfrac{1}{1 - v x^{-1}},
    \]
    which is exactly our claim.
    \end{proof}
\end{example}

Recall that since $f_0(x) = 1$, the sequence $\widecheck{\mathcal{F}}$ is constant-term free. Thus, for any $n \ge 0$, the ratio $\dfrac{\widecheck{f}_{n + 1}(v)}{v}$ is a polynomial.

\begin{definition}
    Given a partition $\nu$ of length $\ell(\nu) \le q$, we define the \emph{double dual Schur function} by 
    \begin{equation}
        \widecheck{s}_\nu^\mathcal{F} (v_1, \dotsc, v_q)
        = 
        \dfrac
        {
            \det \left[ \dfrac{\widecheck{f}_{\nu_j + q - j + 1}(v_i) }{v_i} \right]_{i, j = 1}^q
        }
        {
            \V(v_1, \dotsc, v_q)
        }
        .
        \label{Eq: double dual Schur function}
    \end{equation}
\end{definition}

\begin{example}
    The double dual Schur functions in the factorial case are $s_\nu\factorial{v}{\widetilde{c}}{}.$
\end{example}

\subsection{Ninth variation type A Littlewood identity}

Now we are ready to state our ninth variation of the type A Littlewood identity.

\begin{theorem} \label{Theorem: GL Littlewood}
    The following identity in $K[x_1^{\pm 1}, \dotsc, x_n^{\pm 1}]^{S_n} [[u_1, \dotsc, u_p, v_1, \dotsc, v_q]]$ holds: 
    \begin{equation}
        \begin{gathered}
            \sum_{\lambda = (\mu, \nu)} 
            s_{\lambda - (q^n)}^\mathcal{F} (x_1, \dotsc, x_n) 
            \, 
            \widehat{s}_\mu^\mathcal{F} (u_1, \dotsc, u_p) 
            \, 
            \widecheck{s}_\nu^\mathcal{F} (v_1, \dotsc, v_q) 
            = \\ 
            = 
            \dfrac
            {
                \prod_{i = 1}^{p} \prod_{j = 1}^{q} (1 - u_i v_j)
            }
            {
                \prod_{k = 1}^{n} \prod_{i = 1}^{p} (1 - x_k u_i) 
                \cdot 
                \prod_{k = 1}^{n} \prod_{j = 1}^{q} (x_k - v_j)
            }
            ,
        \end{gathered}
        \label{Eq: GL Littlewood}
    \end{equation}
    where $\lambda - (q^n) = (\lambda_1 - q, \dotsc, \lambda_n - q)$.
\end{theorem}

Before we embark on the proof, note that if one substitutes $f_k(x) = x^k$, the identity~\eqref{Eq: GL Littlewood} transforms into the classical identity~\eqref{Eq: classical GL Littlewood} since $s_{\lambda - (q^n)}(x_1, \dotsc, x_n) = (x_1 \dotsm x_n)^{-q} s_\lambda(x_1, \dotsc, x_n)$ and $\frac{1}{\prod_{k = 1}^{n} \prod_{j = 1}^{q} (x_k - v_j)} = \frac{(x_1 \dotsm x_n)^{-q}}{\prod_{k = 1}^{n} \prod_{j = 1}^{q} (1 - x_k^{-1} v_j)}$.

\begin{proof}
    Consider the following two infinite matrices:
    \begin{align*}
        A = 
        \left[\begin{array}{ccccccc}
        \cdots & f_2(x_1) & f_1(x_1) & f_0(x_1), & f_{-1}(x_1) & f_{-2}(x_1) & \cdots \\
         & \vdots & \vdots & \vdots & \vdots & \vdots &  \\
        \cdots & f_2(x_n) & f_1(x_n) & f_0(x_n), & f_{-1}(x_n) & f_{-2}(x_n) & \cdots
        \end{array}\right], \\
        B = 
        \left[\begin{array}{cccccccc}
        \cdots & \widehat{f}_2(u_1) & \widehat{f}_1(u_1) & \widehat{f}_0(u_1), & 0 & 0 & \cdots \\
         & \vdots & \vdots & \vdots & \vdots & \vdots &  \\
        \cdots & \widehat{f}_2(u_p) & \widehat{f}_1(u_p) & \widehat{f}_0(u_p), & 0 & 0 & \cdots \\
        \cdots & 0 & 0 & 0, & \dfrac{\widecheck{f}_1(v_1)}{v_1} & \dfrac{\widecheck{f}_2(v_1)}{v_1} & \cdots \\
         & \vdots & \vdots & \vdots & \vdots & \vdots &  \\
        \cdots & 0 & 0 & 0, & \dfrac{\widecheck{f}_1(v_q)}{v_q} & \dfrac{\widecheck{f}_2(v_q)}{v_q} & \cdots
        \end{array}\right].
    \end{align*}
    We number their columns by all the integers in reverse order, $\dotsc, 2, 1, 0, -1, -2, \dotsc$, and the commas in the matrices are placed between the $0$-th and $(-1)$-st columns.

    Let us evaluate the determinant $\det A B^t$ in two ways.
    First, one can calculate the entries of the matrix $A B^t$ explicitly. Let $1 \le k \le n$, $1 \le i \le p$, $1 \le j \le q$. We have 
    \begin{gather*}
        (A B^t)_{k, i} 
        = 
        \sum_{m = 0}^\infty 
        f_m(x_k) \widehat{f}_m(u_i)
        = 
        \dfrac{ 1 }{ 1 - x_k u_i }
        , 
        \\ 
        (A B^t)_{k, p + j} 
        = 
        \sum_{m = 1}^\infty 
        f_{-m}(x_k) \dfrac{\widecheck{f}_m(v_j)}{v_j}
        = 
        \dfrac{1}{v_j} \sum_{m = 0}^\infty f_{-m}(x_k) \widecheck{f}_m(v_j) - \dfrac{1}{v_j} f_0(x_k) \widecheck{f}_0(v_j)
        .
    \end{gather*}
    Since $\mathcal{F}$ is admissible, $f_0(x_k) = 1$. It follows that $\widecheck{f}_0(v_j) = 1$. Hence the latter expression becomes
    \[
        \dfrac{1}{v_j} \sum_{m = 0}^\infty 
        f_{-m}(x_k) \widecheck{f}_m(v_j) - \dfrac{1}{v_j} f_0(x_k) \widecheck{f}_0(v_j) 
        = 
        \dfrac{1}{v_j} \left( \dfrac{ 1 }{ 1 - x_k^{-1} v_j } - 1 \right) 
        = 
        \dfrac{1}{x_k - v_j}.
    \]
    Now we have
    \begin{align*}
        A B^t 
        = 
        \left[
        \begin{array}{cccccc}
        \dfrac{ 1 }{ 1 - x_1 u_1 } & \cdots & \dfrac{ 1 }{ 1 - x_1 u_p } & \dfrac{1}{x_1 - v_1} & \cdots & \dfrac{1}{x_1 - v_q} \\
        \vdots &  & \vdots & \vdots &  & \vdots \\
        \vdots &  & \vdots & \vdots &  & \vdots \\
        \dfrac{ 1 }{ 1 - x_n u_1 } & \cdots & \dfrac{ 1 }{ 1 - x_n u_p } & \dfrac{1}{x_n - v_1} & \cdots & \dfrac{1}{x_n - v_q}
        \end{array}
        \right]
        .
    \end{align*}
    The determinant of this matrix can be easily reduced to the Cauchy determinant. In this way we obtain
    \[
        \begin{split}
            \det A B^t 
            = 
            (-1)^{\binom{q}{2}}
            \cdot 
            \dfrac
            {
                \prod_{i = 1}^{p} \prod_{j = 1}^{q} (1 - u_i v_j)
            }
            {
                \prod_{k = 1}^{n} \prod_{i = 1}^{p} (1 - x_k u_i) 
                \cdot 
                \prod_{k = 1}^{n} \prod_{j = 1}^{q} (x_k - v_j)
            } \cdot
            \\
            \cdot \V(x_1, \dotsc, x_n) \V(u_1, \dotsc, u_p) \V(v_1, \dotsc, v_q).
        \end{split}
    \]

    \medskip
    Now we turn to the second way. 
    From the Cauchy--Binet identity it follows that 
    \[
        \det A B^t = \sum_{+ \infty > l_1 > \dotsb > l_n > -\infty} \det A_{l_1, \dotsc, l_n} \det B_{l_1, \dotsc, l_n},
    \]
    where $A_{l_1, \dotsc, l_n}$ and $B_{l_1, \dotsc, l_n}$ are the submatrices of $A$ and $B$ formed by the columns with numbers $l_1, \dotsc, l_n$.
    The determinant $\det B_{l_1, \dotsc, l_n}$ vanishes unless $l_p \ge 0$ and $l_{p + 1} \le -1$, so we can write the sequence $(l_1, \dotsc, l_n)$ as 
    \[
        (l_1, \dotsc, l_n) 
        = 
        (\mu_1, \dotsc, \mu_p, -\nu_q, \dotsc, -\nu_1) + (p - 1, \dotsc, 0, -1, \dotsc, -q),
    \]
    where $\mu$ and $\nu$ are partitions. The determinant is then equal to 
    \[
        \det B_{l_1, \dotsc, l_n} 
        = 
        \widehat{s}_\mu^\mathcal{F} (u_1, \dotsc, u_p) \V(u_1, \dotsc, u_p)
        \cdot
        (-1)^{\binom{q}{2}} 
        \cdot
        \widecheck{s}_{\nu}^\mathcal{F} (v_1, \dotsc, v_q)  \V(v_1, \dotsc, v_q).
    \]
    To handle the second determinant, $\det A_{l_1, \dotsc, l_n}$, denote as before 
    \[
        \lambda = (\lambda_1, \dotsc, \lambda_n) = (\mu_1, \dotsc, \mu_p, -\nu_q, \dotsc, -\nu_1) = (\mu, \nu),
    \]
    and rewrite the sequence $(l_1, \dotsc, l_n)$ in another way:
    \[
        (l_1, \dotsc, l_n) 
        = 
        (\lambda_1 - q, \dotsc, \lambda_n - q) + (n - 1, \dotsc, 0).
    \]
    Now it remains to notice that 
    \[
        \det A_{l_1, \dotsc, l_n} 
        = 
        s_{\lambda - (q^n)}^\mathcal{F} (x_1, \dotsc, x_n) \V(x_1, \dotsc, x_n)
    \]
    This concludes the proof of the theorem.
\end{proof}


\section{Jacobi--Trudi identities}

Recall that for the factorial Schur functions there is a Jacobi--Trudi identity of the following form (see~\cite{Macdonald}):
\[
    s_\lambda \factorial{x}{c}{} = \det \left[ h_{\lambda_i - i + j} \factorial{x}{\tau^{1 - j} c}{} \right],
\]
where $h_k \factorial{x}{\cdot}{} = s_{(k)} \factorial{x}{\cdot}{}$ are the complete factorial symmetric polynomials, and $h_k \factorial{x}{\cdot}{} = 0$ for $k < 0$.

In this section we prove similar Jacobi--Trudi identities for the factorial characters of Foley and King. In the nonfactorial case these identites are due to H.~Weyl~\cite{Weyl}, see also Koike--Terada~\cite{KoikeTerada}. They are also sometimes called \emph{second Weyl identities}, first Weyl identities being the Weyl character formula.

We use the principle that given a Cauchy-like identity, one can expect a Jacobi-Trudi identity to hold, and vice versa. This was exploited by Molev~\cite{Molev} to produce a Cauchy identity for the double Schur functions. Our proof follows along Molev's lines.

\subsection{Generating functions}

It is immediate to deduce generating functions from the Littlewood identities. In what follows we will only need the generating function of $h_k \factorial{x}{c}{}$'s, but we give all four of them for the sake of completeness. 

These identities are taken as definitions of $h_k^{gl} \factorial{x}{c}{} = h_k \factorial{x}{c}{}$, $h_k^{sp} \factorial{x}{c}{} = \sp_{(k)} \factorial{x}{c}{}$, $h_k^{oo} \factorial{x}{c}{} = \so_{(k)} \factorial{x}{c}{}$ and $h_k^{eo} \factorial{x}{c}{} = \o_{(k)} \factorial{x}{c}{}$ in Foley--King~\cite{2018a}.

\begin{proposition}
    We have
    \begin{align}
        & \sum_{k = 0}^\infty h_k \factorial{x}{c}{} \dfrac{t^k}{\prod_{l = 0}^{k + n - 1} (1 - t c_l)} 
        = 
        \dfrac{1}{\prod_{i = 1}^n (1 - x_i t)};
        \label{eq: h_k generating function}
        \\
        & \sum_{k = 0}^\infty \sp_{(k)} \factorial{x}{c}{} \dfrac{t^k}{\prod_{l = 0}^{k + n - 1} (1 - t c_l)} 
        = 
        \dfrac{1}{\prod_{i = 1}^n (1 - x_i t) (1 - \overline{x}_i t)};
        \label{eq: type C generating function}
        \\
        & \sum_{k = 0}^\infty \so_{(k)} \factorial{x}{c}{} \dfrac{t^k}{\prod_{l = 0}^{k + n - 1} (1 - t c_l)} 
        = 
        \dfrac{1 + t}{\prod_{i = 1}^n (1 - x_i t) (1 - \overline{x}_i t)};
        \label{eq: type B generating function}
        \\
        & \sum_{k = 0}^\infty \o_{(k)} \factorial{x}{c}{} \dfrac{t^k}{\prod_{l = 0}^{k + n - 1} (1 - t c_l)} 
        = 
        \dfrac{1 - t^2}{\prod_{i = 1}^n (1 - x_i t) (1 - \overline{x}_i t)}.
        \label{eq: type D generating function}
    \end{align}
\end{proposition}

\begin{proof}
    Set $u_n = 0$ in $\widehat{s}_\lambda \factorial{u_1, \dotsc, u_n}{c}{}$. The denominator becomes 
    \[
        \V (u_1, \dotsc, u_{n - 1}, 0) 
        = 
        u_1 \dotsm u_{n - 1} \V (u_1, \dotsc, u_{n - 1}).
    \]
    If $\lambda_n \ne 0$, the numerator becomes $0$, since the last row of the determinant in the numerator consists entirely of zeros. Otherwise, it turns into
    \[
        \dfrac{u_1 \dotsm u_{n - 1}}{(1 - u_1 c_0) \dotsm (1 - u_{n - 1} c_0)}
        \det 
        \left[ 
            \dfrac{u_i^{\lambda_j + (n - 1) - j}}{ \prod_{l = 0}^{\lambda_j + (n - 1) - j} (1 - u_i c_{l + 1}) }
        \right]_{i, j = 1}^{n - 1}.
    \]
    Hence 
    \[
        \widehat{s}_\lambda \factorial{u_1, \dotsc, u_n}{c}{} \, |_{u_n = 0}
        = 
        \delta_{\lambda_n, 0}
        \dfrac{1}{(1 - u_1 c_0) \dotsm (1 - u_{n - 1} c_0)}
        \widehat{s}_\lambda \factorial{u_1, \dotsc, u_{n - 1}}{\tau c}{}.
    \]
    
    Now let us successively substitute $u_{n - 1} = 0$, $u_{n - 2} = 0$, \ldots, $u_{2} = 0$. Eventually, we are left with 
    \[
        \delta_{\lambda_n, 0} \dotsm \delta_{\lambda_2, 0} \,
        \dfrac{1}{\prod_{l = 0}^{n - 2} (1 - u_1 c_l)}
        \widehat{s}_\lambda \factorial{u_1}{\tau^{n - 1} c}{}.
    \]
    Since $\widehat{s}_\lambda \factorial{u_1}{\tau^{n - 1} c}{} = \dfrac{u_1^{\lambda_1}}{\prod_{l = 0}^{\lambda_1} (1 - u_1 c_{n - 1 + l})}$, this is equal to 
    \[
        \delta_{\lambda_n, 0} \dotsm \delta_{\lambda_2, 0} \,
        \dfrac{u_1^{\lambda_1}}{\prod_{l = 0}^{ \lambda_1 + n - 1 } (1 - u_1 c_l)}.
    \]

    In order to obtain the generating functions~\eqref{eq: type C generating function},~\eqref{eq: type B generating function},~\eqref{eq: type D generating function} in the types CBD, take any of the Littlewood identites~\eqref{Eq: Sp Littlewood},~\eqref{Eq: SO(2n + 1) Littlewood},~\eqref{Eq: SO(2n) Littlewood} and substitute $u_n = u_{n - 1} = \dotsb = u_2 = 0$. Only the terms with $\lambda_n = \dotsb = \lambda_2 = 0$ survive, and we arrive at the generating function identities stated in the theorem.

    To produce formula~\eqref{eq: h_k generating function}, take the type A Littlewood identity~\eqref{Eq: GL Littlewood} with $q = 0$ (i.e.~the Cauchy identity) and apply the same procedure.
\end{proof}

    We will only need the type A generating function~\eqref{eq: h_k generating function}, and it will be more convinient to have it in the following form:

\begin{corollary}
    We have 
    \begin{equation}
        \sum_{k = 0}^\infty h_k \factorial{x}{c}{} \dfrac{t^k}{\prod_{l = 1}^{k} (1 - t c_{n - 1 + l})} 
        = 
        \dfrac{ \prod_{l = 0}^{n - 1} (1 - t c_l) }{\prod_{i = 1}^n (1 - x_i t)}.
        \label{eq: h_k convenient generating function}
    \end{equation}
\end{corollary}

\begin{remark}
    Recall that $S^k(\mathbb{C}^{2n})$ is an irreducible module over $\Sp(2n, \mathbb{C})$. That is, 
    \[
        \sp_{(k)} (x) = h_k (x, \overline{x}).
    \]
    Thus, we would expect that there is a similar formula for the factorial characters $\sp_{(k)} \factorial{x}{c}{}$. In other words, we are looking for such a sequence $a$ that 
    \[
        \sp_{(k)} \factorial{x}{c}{} = h_k \factorial{x, \overline{x}}{a}{}.
    \]

    It follows from~\eqref{eq: h_k generating function} and~\eqref{eq: type C generating function} that the only natural way to choose such a sequence is 
    \[
        a = \tau^{-n} ( c^{\text{cut}} ),
    \]
    where $(c^{\text{cut}})_m = 0$ for $m < 0$ and $(c^{\text{cut}})_m = c_m$ for $m \ge 0$. This justifies the appearance of negative shifts in the Jacobi--Trudi identities below.
\end{remark}

\subsection{Types C, B, D}

We begin with the type C.

\begin{theorem}
    Assume that $c_n = 0$ for all $n < 0$. Then we have the following identity:
    \begin{equation}
        \sp_\lambda \factorial{x}{c}{} = 
        \dfrac{1}{2} 
        \det \left[ h_{\lambda_i - i + j} \factorial{x, \overline{x}}{\tau^{1 - n - j} c}{} + h_{\lambda_i - i - j + 2} \factorial{x, \overline{x}}{\tau^{-1 - n + j} c}{} \right]_{i, j = 1}^n,
        \label{eq: factorial JT C}
    \end{equation}
    where $h_k \factorial{x, \overline{x}}{\cdot}{}$ are the complete factorial symmetric polynomials in the $2n$ variables $x_1, \dotsc, x_n, \overline{x}_1, \dotsc, \overline{x}_n$, and $h_k \factorial{x, \overline{x}}{\cdot}{} = 0$ for $k < 0$.
    \label{Theorem: Jacobi-Trudi in type C}
\end{theorem}

\begin{remark}
    One could also rewrite the identity~\eqref{eq: factorial JT C} more succinctly using the sequence $a = \tau^{-n} ( c^{\text{cut}} )$ from above.
\end{remark}

\begin{proof}
    For each $\alpha = (\alpha_1, \dotsc, \alpha_n) \in \mathbb{N}^n$ define 
    \begin{equation*} 
        \jt_\alpha^C \factorial{x}{c}{} 
        = 
        \dfrac{1}{2} 
        \det \left[ h_{\alpha_i - n + j} \factorial{x, \overline{x}}{\tau^{1 - n - j} c}{} + h_{\alpha_i - n - j + 2} \factorial{x, \overline{x}}{\tau^{-1 - n + j} c}{} \right]_{i, j = 1}^n.
    \end{equation*}
    Then on the right-hand side of~\eqref{eq: factorial JT C} we have $\jt_{\lambda + \rho}^C \factorial{x}{c}{}$, where $\rho = (n - 1, \dotsc, 0)$. 
    
    Consider the sum
    \begin{equation} \label{eq: jt c}
        \sum_{\ell(\lambda) \le n} 
        \jt_{\lambda + \rho}^C \factorial{x}{c}{}
        \, 
        \widehat{s}_\lambda \factorial{u}{c}{}.
    \end{equation}
    As follows from the Littlewood identity, the theorem is proven if we show that this sum is equal to
    \[
        \dfrac
        {
            \prod_{1 \le i < j \le n} (1 - u_i u_j)
        }
        {
            \prod_{i, j = 1}^n (1 - x_i u_j) (1 - \overline{x}_i u_j)
        }.
    \]
    Multiply the sum~\eqref{eq: jt c} by the Vandermonde $\V(u_1, \dotsc, u_n)$. Then it turns into
    \begin{equation} \label{eq: jtc and det gamma}
        \sum_{\gamma} 
        \jt_{\gamma}^C \factorial{x}{c}{}
        \, 
        \det \left[ \dfrac{u_i^{\gamma_j}}{ \prod_{l = 0}^{\gamma_j} (1 - u_i c_l) } \right],
    \end{equation}
    where the sum is taken over $n$-tuples $\gamma = (\gamma_1, \dotsc, \gamma_n) \in \mathbb{N}^n$ with $\gamma_1 > \dotsb > \gamma_n \ge 0$.

    Since 
    \[
        \det \left[ \dfrac{u_i^{\gamma_j}}{ \prod_{l = 0}^{\gamma_j} (1 - u_i c_l) } \right] 
        = 
        \sum_{\sigma \in S_n} \sgn(\sigma) 
        \prod_{i = 1}^n 
        \dfrac{u_i^{\gamma_{\sigma(i)}}}{ \prod_{l = 0}^{\gamma_{\sigma(i)}} (1 - u_i c_l) }
    \]
    and $\jt_{\gamma}^C \factorial{x}{c}{} $ is skew-symmetric under permutations of the components of $\gamma$, we can rewrite~\eqref{eq: jtc and det gamma} in the form
    \[
        \sum_{\alpha} 
        \jt_{\alpha}^C \factorial{x}{c}{}
        \, 
        \prod_{i = 1}^n \dfrac{u_i^{\alpha_i}}{ \prod_{l = 0}^{\alpha_i} (1 - u_i c_l) },
    \]
    where the sum is taken over all $n$-tuples $\alpha = (\alpha_1, \dotsc, \alpha_n) \in \mathbb{N}^n$ of nonnegative integers.
    Using the definition of $\jt_{\alpha}^C \factorial{x}{c}{}$ we can rewrite the sum once again:
    \[
        \dfrac{1}{2} 
        \sum_{\alpha} 
        \det \left[ \left( h_{\alpha_i - n + j} \factorial{x, \overline{x}}{\tau^{1 - n - j} c}{} + h_{\alpha_i - n - j + 2} \factorial{x, \overline{x}}{\tau^{-1 - n + j} c}{} \right) \dfrac{u_i^{\alpha_i}}{ \prod_{l = 0}^{\alpha_i} (1 - u_i c_l) } \right].
    \]
    Writing the determinant as a sum over the symmetric group we get 
    \[
        \begin{split}
            \dfrac{1}{2} 
            \sum_{\alpha} 
            \sum_{\sigma \in S_n} \sgn(\sigma)
            \prod_{j = 1}^n & 
            \left( h_{\alpha_{\sigma(j)} - n + j} \factorial{x, \overline{x}}{\tau^{1 - n - j} c}{} + h_{\alpha_{\sigma(j)} - n - j + 2} \factorial{x, \overline{x}}{\tau^{-1 - n + j} c}{} \right) 
            \cdot
            \\ &
            \cdot
            \dfrac{u_{\sigma(j)}^{\alpha_{\sigma(j)}}}{ \prod_{l = 0}^{\alpha_{\sigma(j)}} (1 - u_{\sigma(j)} c_l) }.
        \end{split}
    \]
    Interchanging the order of the sums we obtain 
    \begin{equation} \label{eq: jtc long sum}
        \begin{split}
           \dfrac{1}{2} 
           \sum_{\sigma \in S_n} \sgn(\sigma)
           \prod_{j = 1}^n 
           \sum_{k = 0}^\infty & 
           \left( h_{k - n + j} \factorial{x, \overline{x}}{\tau^{1 - n - j} c}{} + h_{k - n - j + 2} \factorial{x, \overline{x}}{\tau^{-1 - n + j} c}{} \right) 
           \cdot
           \\ &
           \cdot
           \dfrac{u_{\sigma(j)}^{k}}{ \prod_{l = 0}^{k} (1 - u_{\sigma(j)} c_l) }.
        \end{split}
    \end{equation}
    Now it is left to calculate the following:
    \[
        \begin{gathered}
            \sum_{k = 0}^\infty 
            \left( 
                h_{k - n + j} \factorial{x, \overline{x}}{\tau^{1 - n - j} c}{} + h_{k - n - j + 2} \factorial{x, \overline{x}}{\tau^{-1 - n + j} c}{} 
            \right) 
            \dfrac{u^{k}}{ \prod_{l = 0}^{k} (1 - u c_l) } 
            = \\ = 
            \sum_{k = 0}^\infty 
            h_k \factorial{x, \overline{x}}{\tau^{1 - n - j} c}{} 
            \dfrac{u^{k + n - j}}{ \prod_{l = 0}^{k + n - j} (1 - u c_l) } 
            + 
            \sum_{k = 0}^\infty 
            h_k \factorial{x, \overline{x}}{\tau^{-1 - n + j} c}{} 
            \dfrac{u^{k + n + j - 2}}{ \prod_{l = 0}^{k + n + j - 2} (1 - u c_l) }.
        \end{gathered}
    \]
    The first sum is equal to 
    \[
        \dfrac{u^{n - j}}{ \prod_{l = 0}^{n - j} (1 - u c_l) }
        \cdot
        \sum_{k = 0}^\infty 
        h_k \factorial{x, \overline{x}}{\tau^{1 - n - j} c}{} 
        \dfrac{u^k}{ \prod_{l = 1}^{k} (1 - u c_{n - j + l}) }.
    \]
    Recalling the generating function of the complete factorial symmetric polynomials~\eqref{eq: h_k convenient generating function} and noticing that the number of variables in the $h_k$'s is $2n$, we get 
    \[
        \dfrac{u^{n - j}}{ \prod_{l = 0}^{n - j} (1 - u c_l) }
        \cdot
        \dfrac
        {
            \prod_{l = 0}^{2n - 1} (1 - u c_{1 - n - j + l})
        }
        {
            \prod_{i = 1}^n (1 - x_i u) (1 - \overline{x}_i u)
        }.
    \]
    Using the assumption $c_n = 0$ for $n < 0$, we are left with only 
    \[
        \dfrac{ u^{n - j} }{\prod_{i = 1}^n (1 - x_i u) (1 - \overline{x}_i u)}.
    \]
    The second sum is dealt with similarly: 
    \[
        \begin{gathered}
            \dfrac{u^{n + j - 2}}{ \prod_{l = 0}^{n + j - 2} (1 - u c_l) }
            \cdot
            \sum_{k = 0}^\infty 
            h_k \factorial{x, \overline{x}}{\tau^{-1 - n + j} c}{} 
            \dfrac{u^k}{ \prod_{l = 1}^{k} (1 - u c_{n + j - 2 + l}) } 
            = \\ =
            \dfrac{ u^{n + j - 2} }{\prod_{i = 1}^n (1 - x_i u) (1 - \overline{x}_i u)}.
        \end{gathered}
    \]
    Substituting these into~\eqref{eq: jtc long sum}, we obtain
    \begin{equation}
        \begin{gathered}
            \dfrac{1}{2} 
            \sum_{\sigma \in S_n} \sgn(\sigma)
            \prod_{j = 1}^n 
            \dfrac{ u_{\sigma(j)}^{n - j} + u_{\sigma(j)}^{n + j - 2} }{\prod_{i = 1}^n (1 - x_i u_{\sigma(j)}) (1 - \overline{x}_i u_{\sigma(j)})} 
            = \\ =
            \dfrac{ 1 }{\prod_{i, j = 1}^n (1 - x_i u_j) (1 - \overline{x}_i u_j) } 
            \cdot 
            \dfrac{1}{2} 
            \sum_{\sigma \in S_n} \sgn(\sigma)
            \prod_{j = 1}^n 
            \left(
            u_{\sigma(j)}^{n - j} + u_{\sigma(j)}^{n + j - 2}
            \right).
        \end{gathered}
        \label{eq: jtc final}
    \end{equation}
    Here we see the determinant $\det \left[ u_i^{n - j} + u_i^{n + j - 2} \right]$, which can be easily computed through reduction to the type D Weyl denominator. The result is 
    \[
        \det \left[ u_i^{n - j} + u_i^{n + j - 2} \right] 
        = 
        2 \V(u_1, \dotsc, u_n) \prod_{i < j} (1 - u_i u_j).
    \]
    Thus~\eqref{eq: jtc final} transforms into 
    \[
        \V(u_1, \dotsc, u_n)
        \cdot
        \dfrac
        {
            \prod_{1 \le i < j \le n} (1 - u_i u_j)
        }
        {
            \prod_{i, j = 1}^n (1 - x_i u_j) (1 - \overline{x}_i u_j)
        },
    \]
    which is exactly what we awaited.
\end{proof}

The Jacobi--Trudi identites in the types B and D are proven in exactly the same fashion, so we omit the details.

\begin{theorem}
    Assume that $c_n = 0$ for all $n < 0$. Then we have the following identity:
    \begin{equation}
        \so_\lambda \factorial{x}{c}{} = 
        \det 
        \left[ 
            h_{\lambda_i - i + j} \factorial{x, \overline{x}, 1}{\tau^{- n - j} c}{} 
            - 
            h_{\lambda_i - i - j} \factorial{x, \overline{x}, 1}{\tau^{- n + j} c}{} 
        \right]_{i, j = 1}^n,
    \end{equation}
    where $h_k \factorial{x, \overline{x}, 1}{\cdot}{}$ are the complete factorial symmetric polynomials in the $2n + 1$ variables $x_1, \dotsc, x_n, \overline{x}_1, \dotsc, \overline{x}_n, 1$, and $h_k \factorial{x, \overline{x}, 1}{\cdot}{} = 0$ for $k < 0$.
    \label{Theorem: Jacobi-Trudi in type B}
\end{theorem}

In order to make use of the type D Littlewood identity, we need the condition that the sequence $\mathcal{F}$ is constant-term free. In the case of factorial characters this is equivalent to $c_0 = 0$.

\begin{theorem}
    Assume that $c_n = 0$ for all $n < 0$ and also $c_0 = 0$. Then we have the following identity:
    \begin{equation}
        \o_\lambda \factorial{x}{c}{} = 
        \det 
        \left[ 
            h_{\lambda_i - i + j} \factorial{x, \overline{x}}{\tau^{1 - n - j} c}{} 
            - 
            h_{\lambda_i - i - j} \factorial{x, \overline{x}}{\tau^{1 - n + j} c}{} 
        \right]_{i, j = 1}^n,
    \end{equation}
    where $h_k \factorial{x, \overline{x}}{\cdot}{}$ are the complete factorial symmetric polynomials in the $2n$ variables $x_1, \dotsc, x_n, \overline{x}_1, \dotsc, \overline{x}_n$, and $h_k \factorial{x, \overline{x}}{\cdot}{} = 0$ for $k < 0$.
    \label{Theorem: Jacobi-Trudi in type D}
\end{theorem}

\subsection{Classical type A case}

Let $\lambda = (\lambda_1, \dotsc, \lambda_n)$ be a signature. As above, decompose it into a pair of partitions $\lambda = (\mu, \nu)$, where $\mu$ and $\nu$ are of lengths $\ell(\mu) \le p$ and $\ell(\nu) \le q$, and $p + q = n$.

There is the following analogue of the Jacobi--Trudi identity for the rational Schur function $s_\lambda(x_1, \dotsc, x_n)$; it was postulated by Balantekin and Bars~\cite{Balantekin_Bars} in dual form (i.e. N\"{a}gelsbach--Kostka) and proven by Cummins and King~\cite{Cummins_King}, see also Koike~\cite[Proposition~2.8]{Koike}. Let us denote 
\[
    h_k^* (x_1, \dotsc, x_n) = s_{(k, 0, \dotsc, 0)}(x_1^{-1}, \dotsc, x_n^{-1}),
\]
i.e.~the character of $S^k(\mathbb{C}^n)^*$, where $\mathbb{C}^n$ is the tautological representation of $\GL(n, \mathbb{C})$. In other words, 
\[
    h_k^* (x_1, \dotsc, x_n) = s_{(k, 0, \dotsc, 0)^*}(x_1, \dotsc, x_n),
\]
where $(k, 0, \dotsc, 0)^* = (0, \dotsc, 0, -k)$ is the dual of the highest weight $(k, 0, \dotsc, 0)$. 
Set also $h_k^* (x_1, \dotsc, x_n) = 0$ for negative $k$. Then
\begin{equation}
    s_\lambda (x_1, \dotsc, x_n) 
    = 
    \det 
    \left[\begin{array}{c}
        h_{\nu_{q - i + 1} + i - j}^* (x_1, \dotsc, x_n) \\
        \hdashline
        h_{\mu_i - q - i + j} (x_1, \dotsc, x_n)
    \end{array}\right],
\end{equation}
where on the right we have a block matrix consisting of two blocks of sizes $q \times n$ and $p \times n$. The indices in the upper block run through $i = 1, \dotsc, q$ and $j = 1, \dotsc, n$, and in the lower block through $i = 1, \dotsc, p$ and $j = 1, \dotsc, n$.

We aim for a similar identity for the factorial Schur functions $s_\lambda \factorial{x}{c}{}$. We will replace $h_k^* (x_1, \dotsc, x_n)$ with $g_k \factorial{x}{c}{}$, where we define
    \begin{equation}
        g_k \factorial{x}{c}{} 
        = 
        \begin{cases}
            0, & \mbox{\textup{if} $k < 0$}; \\
            s_{(k, 0, \dotsc, 0)^* - (1^n)} \factorial{x}{c}{}, & \mbox{\textnormal{if} $k \ge 0$}.
        \end{cases}
    \end{equation}
We note, however, that $g_k \factorial{x}{c}{}$ does not specialise into $h_k^* (x_1, \dotsc, x_n)$ if we set $c_m = 0$ for all $m \in \mathbb{Z}$, but rather into $(x_1 \dotsm x_n)^{-1} h_k^* (x_1, \dotsc, x_n)$.

\subsection{Generating function}

First, we need the generating function of the $g_k \factorial{x}{c}{}$'s. Recall the reversed sequence $\widetilde{c}$ from Example~\ref{Example: double dual} we defined by $\widetilde{c}_n = c_{-n - 1}$, $n \in \mathbb{Z}$.

\begin{proposition}
    We have 
    \begin{equation} \label{eq: dual generating function}
        \sum_{k = 0}^\infty 
        g_k \factorial{x}{c}{} 
        \,
        \factorial{ t }{ \widetilde{c} }{ k } 
        = 
        \dfrac{ 1 }{ \prod_{i = 1}^{n} (x_i - t) }.
    \end{equation}
\end{proposition}

\begin{proof}
    Consider the type A Littlewood identity~\eqref{Eq: GL Littlewood} with $p = 0$:
    \begin{equation} \label{eq: dual cauchy}
            \sum_{\lambda = (\emptyset, \nu)} 
            s_{\lambda - (n^n)} \factorial{x_1, \dotsc, x_n}{c}{} 
            \, 
            s_\nu \factorial{v_1, \dotsc, v_n}{\widetilde{c}}{} 
            = 
            \dfrac{ 1 }{ \prod_{i = 1}^{n} \prod_{j = 1}^{n} (x_i - v_j) }.
    \end{equation}
    On the left we have $\lambda = (\emptyset, \nu) = \nu^*$, and $\nu$ ranges over the set of all partitions of length $\ell(\nu) \le n$.

    Set $v_n = \widetilde{c}_0 = c_{-1}$. It is easy to see that 
    \[
        s_\nu \factorial{v_1, \dotsc, v_n}{\widetilde{c}}{} \, |_{v_n = \widetilde{c}_0}
        = 
        \delta_{\nu_n, 0} \, 
        s_\nu \factorial{v_1, \dotsc, v_{n - 1}}{\tau \widetilde{c}}{}.
    \]
    Set further $v_{n - 1} = \widetilde{c}_1 = c_{-2}$, \ldots, $v_2 = \widetilde{c}_{n - 2} = c_{- n + 1}$. Then we get 
    \[
        \delta_{\nu_n, 0} \dotsm \delta_{\nu_2, 0} \, 
        s_\nu \factorial{v_1}{\tau^{n - 1} \widetilde{c}}{},
    \]
    which is equal to
    \[
        \delta_{\nu_n, 0} \dotsm \delta_{\nu_2, 0} \,
        \factorial{v_1}{\tau^{n - 1} \widetilde{c}}{\nu_1}.
    \]
    Hence on the left of the equation~\eqref{eq: dual cauchy} we see 
    \[
        \sum_{k = 0}^\infty
        s_{(k, 0, \dotsc, 0)^* - (n^n)} \factorial{x_1, \dotsc, x_n}{c}{} 
        \factorial{v_1}{\tau^{n - 1} \widetilde{c}}{ k }.
    \]
    At the same time, on the right we have 
    \[
        \dfrac{ 1 }
        { 
            \prod_{i = 1}^{n} (x_i - v_1) \cdot \prod_{i = 1}^{n} \factorial{x_i}{\widetilde{c}}{n - 1} 
        }.
    \]
    Now note that 
    \[
        s_{(k, 0, \dotsc, 0)^* - (n^n)} \factorial{x_1, \dotsc, x_n}{c}{} 
        = 
        \prod_{i = 1}^n \factorial{x_i}{c}{ -n + 1 } 
        \cdot 
        s_{(k, 0, \dotsc, 0)^* - (1^n)} \factorial{x_1, \dotsc, x_n}{\tau^{-n + 1} c}{},
    \]
    and we can rewrite this as 
    \[
        \dfrac{ 1 }
        { 
            \prod_{i = 1}^{n} \factorial{x_i}{\widetilde{c}}{n - 1} 
        }
        \,
        g_k \factorial{x}{\tau^{-n + 1} c}{}.
    \]
    Thus~\eqref{eq: dual cauchy} implies, letting $t = v_1$, that 
    \[
        \sum_{k = 0}^\infty 
        g_k \factorial{x}{\tau^{-n + 1} c}{} 
        \,
        \factorial{ t }{ \tau^{n - 1} \widetilde{c} }{ k } 
        = 
        \dfrac{ 1 }{ \prod_{i = 1}^{n} (x_i - t) }.
    \]
    Now it remains to notice that $\tau^{n - 1} \widetilde{c} = \widetilde{\tau^{-n + 1} c}$, and the theorem follows.
\end{proof}

\subsection{Jacobi--Trudi identity}

Now we are ready to proceed with the Jacobi--Trudi identity.

\begin{theorem}
    Decompose the signature $\lambda$ as $\lambda = (\mu, \nu)$. Then we have
    \begin{equation}
        s_{\lambda - (q^n)} \factorial{x}{c}{} 
        = 
        \det 
        \left[\begin{array}{c}
            g_{\nu_{q - i + 1} + i - j} \factorial{x}{\tau^{1 - j} c}{} \\
            \hdashline
            h_{\mu_i - q - i + j} \factorial{x}{\tau^{1 - j} c}{}
        \end{array}\right],
        \label{eq: factorial JT A}
    \end{equation}
    where in the upper block $i = 1, \dotsc, q$, in the lower block $i = 1, \dotsc, p$, and $j = 1, \dotsc, n$ in both blocks.
    \label{Theorem: Jacobi-Trudi in type A}
\end{theorem}

\begin{proof}
    For an arbitrary sequence $\varepsilon = (\varepsilon_1, \dotsc, \varepsilon_n) \in \mathbb{Z}^n$ define
    \begin{equation} \label{eq: jt A def}
        \jt_\varepsilon^A \factorial{x}{c}{} 
        = 
        \det
        \left[\begin{array}{c}
            g_{ - \varepsilon_{p + i} + q - j } \factorial{x}{\tau^{1 - j} c}{} \\
            \hdashline
            h_{ \varepsilon_i - n - q + j } \factorial{x}{\tau^{1 - j} c}{}
        \end{array}\right],
    \end{equation}
    where in the upper block $i = 1, \dotsc, q$, in the lower block $i = 1, \dotsc, p$, and $j = 1, \dotsc, n$ in both blocks.
    Then on the right-hand side of~\eqref{eq: factorial JT A} we have $\jt_{\lambda + \rho}^A \factorial{x}{c}{}$, where $\rho = (n - 1, \dotsc, 0)$.

    Due to the Littlewood identity, we want to prove that the sum
    \begin{equation} \label{eq: jt A}
        \sum_{\lambda = (\mu, \nu)} 
        \jt_{\lambda + \rho}^A \factorial{x}{c}{} 
        \, 
        \widehat{s}_\lambda \factorial{u}{c}{} 
        \, 
        s_\nu\factorial{v}{\widetilde{c}}{}
    \end{equation}
    is equal to
    \[
        \dfrac
        {
            \prod_{i = 1}^{p} \prod_{j = 1}^{q} (1 - u_i v_j)
        }
        {
            \prod_{k = 1}^{n} \prod_{i = 1}^{p} (1 - x_k u_i) 
            \cdot 
            \prod_{k = 1}^{n} \prod_{j = 1}^{q} (x_k - v_j)
        }.
    \]
    First, multiply the sum~\eqref{eq: jt A} by the Vandermondes $\V(u_1, \dotsc, u_p)$ and $\V(v_1, \dotsc, v_q)$.
    Then it becomes
    \begin{equation} \label{eq: jt A and two dets}
        \sum_{\gamma, \delta} 
        \jt_{(\gamma_1 + q, \dotsc, \gamma_p + q, - \delta_q + q - 1, \dotsc, -\delta_1 + q - 1)}^A \factorial{x}{c}{}
        \, 
        \det 
        \left[ 
            \dfrac{u_i^{\gamma_j}}{ \prod_{l = 0}^{\gamma_j} (1 - u_i c_l) } 
        \right]_{i, j = 1}^p
        \mathllap{\det} 
        \left[ 
            \factorial{v_i}{\widetilde{c}}{\delta_j} 
        \right]_{i, j = 1}^q,
    \end{equation}
    where the sum is taken over all the tuples $\gamma \in \mathbb{N}^p$ and $\delta \in \mathbb{N}^q$ with $\gamma_1 > \dotsb > \gamma_p \ge 0$ and $\delta_1 > \dotsb > \delta_q \ge 0$.

    Writing the two determinants as sums over symmetric groups and using skew-symmetry of $\jt_\varepsilon^A \factorial{x}{c}{}$, we can rewrite the sum~\eqref{eq: jt A and two dets} as 
    \[
        \sum_{\alpha, \beta} 
        \jt_{(\alpha_1 + q, \dotsc, \alpha_p + q, - \beta_q + q - 1, \dotsc, -\beta_1 + q - 1)}^A \factorial{x}{c}{}
        \, 
        \prod_{i = 1}^p
            \dfrac{u_i^{\alpha_i}}{ \prod_{l = 0}^{\alpha_i} (1 - u_i c_l) }
        \,
        \prod_{j = 1}^q 
            \factorial{v_j}{\widetilde{c}}{\beta_j},
    \]
    where the sum is taken over all tuples $\alpha \in \mathbb{N}^p$ and $\beta \in \mathbb{N}^q$ of nonnegative integers.
    Now we can introduce the factors on the right into the determinant~\eqref{eq: jt A def}:
    \[
        \sum_{\alpha, \beta} 
        \det
        \left[\begin{array}{c}
            \\[-1.5ex]
            g_{ \beta_{q - i + 1}  - j + 1 } \factorial{x}{\tau^{1 - j} c}{} \factorial{v_{q - i + 1}}{\widetilde{c}}{\beta_{q - i + 1}} \\
            \\[-1ex]
            \hdashline
            \\[-1.5ex]
            h_{ \alpha_i - n + j } \factorial{x}{\tau^{1 - j} c}{} \frac{u_i^{\alpha_i}}{ \prod_{l = 0}^{\alpha_i} (1 - u_i c_l) }
        \end{array}\right].
    \]
    Now we can expand the determinant as a sum over the symmetric group:
    \[
        \begin{gathered}
        \sum_{\alpha, \beta} 
        \sum_{\sigma \in S_n}
        \sgn(\sigma)
        \Biggl(
            \prod_{ \substack{j \in \{1, \dotsc, n\}: \\ 1 \le \sigma(j) \le q} }
            g_{ \beta_{q - \sigma(j) + 1}  - j + 1 } \factorial{x}{\tau^{1 - j} c}{} 
            \factorial{v_{q - \sigma(j) + 1}}{\widetilde{c}}{\beta_{q - \sigma(j) + 1}}
            \cdot \\ \cdot
            \prod_{ \substack{j \in \{1, \dotsc, n\}: \\ q + 1 \le \sigma(j) \le p + q} }
            h_{ \alpha_{\sigma(j) - q} - n + j } \factorial{x}{\tau^{1 - j} c}{} 
            \frac{u_{\sigma(j) - q}^{\alpha_{\sigma(j) - q}}}{ \prod_{l = 0}^{\alpha_{\sigma(j) - q}} \left( 1 - u_{\sigma(j) - q} c_l \right) }
        \Biggr).
        \end{gathered}
    \]
    Reordering the sums and carrying them inside the brackets we arrive at
    \begin{equation} \label{eq: jt A long sum}
        \begin{gathered}
            \sum_{\sigma \in S_n}
            \sgn(\sigma)
            \Biggl(
                \prod_{ \substack{j \in \{1, \dotsc, n\}: \\ 1 \le \sigma(j) \le q} }
                \sum_{k = 0}^\infty
                g_{ k - j + 1 } \factorial{x}{\tau^{1 - j} c}{} 
                \factorial{v_{q - \sigma(j) + 1}}{\widetilde{c}}{k}
                \cdot \\ \cdot
                \prod_{ \substack{j \in \{1, \dotsc, n\}: \\ q + 1 \le \sigma(j) \le p + q} }
                \sum_{k = 0}^\infty
                h_{ k - n + j } \factorial{x}{\tau^{1 - j} c}{} 
                \frac{u_{\sigma(j) - q}^k}{ \prod_{l = 0}^k \left( 1 - u_{\sigma(j) - q} c_l \right) }
            \Biggr).
            \end{gathered}
    \end{equation}
    Now let us calculate the two sums inside~\eqref{eq: jt A long sum}. The first is equal to 
    \[
        \begin{gathered}
            \sum_{k = 0}^\infty
                g_{ k - j + 1 } \factorial{x}{\tau^{1 - j} c}{} 
                \factorial{v}{\widetilde{c}}{k}
            =
            \sum_{k = 0}^\infty
                g_k \factorial{x}{\tau^{1 - j} c}{}
                \factorial{v}{\widetilde{c}}{k + j - 1}
            = \\ =
            \factorial{v}{\widetilde{c}}{j - 1}
            \sum_{k = 0}^\infty
                g_k \factorial{x}{\tau^{1 - j} c}{}
                \factorial{v}{\tau^{j - 1} \widetilde{c}}{k}
            = 
            \factorial{v}{\widetilde{c}}{j - 1}
            \dfrac{ 1 }{ \prod_{k = 1}^n (x_k - v) },
        \end{gathered}
    \]
    where in the last equality we use~\eqref{eq: dual generating function} together with the equality $\widetilde{\tau^{1 - j} c} = \tau^{j - 1} \widetilde{c}$.
    The second sum equals
    \[
        \begin{gathered}
            \sum_{k = 0}^\infty
                h_{ k - n + j } \factorial{x}{\tau^{1 - j} c}{} 
                \frac{u^k}{ \prod_{l = 0}^k \left( 1 - u c_l \right) }
            = \\ =
            \dfrac{u^{n - j}}{ \prod_{l = 0}^{n - j} (1 - u c_l) }
            \cdot
            \sum_{k = 0}^\infty
                h_k \factorial{x}{\tau^{1 - j} c}{} 
                \frac{u^k}{ \prod_{l = 1}^k \left( 1 - u c_{n - j + l} \right) }
            = \\ =
            u^{n - j} \prod_{s = 1}^{ j - 1 } (1 - u c_{s - j})
            \cdot
            \prod_{k = 1}^n \dfrac{1}{1 - x_k u}.
        \end{gathered}
    \]
    It follows that~\eqref{eq: jt A long sum} is equal to 
    \begin{equation} \label{eq: jt A final}
        \begin{gathered}
            \dfrac{1}{ \prod_{k = 1}^{n} \prod_{j = 1}^{q} (x_k - v_j) }
            \cdot
            \dfrac{1}{ \prod_{k = 1}^{n} \prod_{i = 1}^{p} (1 - x_k u_i) }
            \cdot
            \sum_{\sigma \in S_n}
            \sgn(\sigma)
            \Biggl(
                \prod_{ \substack{j \in \{1, \dotsc, n\}: \\ 1 \le \sigma(j) \le q} }
                \factorial{v_{q - \sigma(j) + 1}}{\widetilde{c}}{j - 1}
                \cdot \\ \cdot
                \prod_{ \substack{j \in \{1, \dotsc, n\}: \\ q + 1 \le \sigma(j) \le p + q} }
                u_{\sigma(j) - q}^{n - j} \prod_{s = 1}^{ j - 1 } (1 - u_{\sigma(j) - q} c_{s - j})
            \Biggr).
        \end{gathered}
    \end{equation}
    We can see that the following determinant appeared:
    \[
        \det
        \left[\begin{array}{c}
            \factorial{v_{q - i + 1}}{\widetilde{c}}{j - 1} 
            \\[1ex]
            \hdashline
            \\[-2ex]
            u_i^{n - j} \prod_{s = 1}^{ j - 1 } (1 - u_i c_{s - j})
        \end{array}\right],
    \]
    where as usual in the upper block $i = 1, \dotsc, q$, in the lower block $i = 1, \dotsc, p$, and $j = 1, \dotsc, n$ in both blocks. Now notice that $u_i^{n - j} \prod_{s = 1}^{ j - 1 } (1 - u_i c_{s - j}) = u_i^{n - 1} \factorial{u_i^{-1}}{\widetilde{c}}{j - 1}$. Hence our determinant is 
    \[
        \det
        \left[\begin{array}{c}
            \factorial{v_{q - i + 1}}{\widetilde{c}}{j - 1} 
            \\[1ex]
            \hdashline
            \\[-2ex]
            u_i^{n - 1} \factorial{u_i^{-1}}{\widetilde{c}}{j - 1}
        \end{array}\right],
    \]
    whence we see that it is independent of $\widetilde{c}$. It is easy to see that
    \[
        \det
        \left[\begin{array}{c}
            v_{q - i + 1}^{j - 1} 
            \\[1ex]
            \hdashline
            \\[-2ex]
            u_i^{n - j}
        \end{array}\right]
        =
        \V(u_1, \dotsc, u_p)
        \V(v_1, \dotsc, v_q)
        \prod_{i = 1}^{p} \prod_{j = 1}^{q} (1 - u_i v_j).
    \]
    Therefore,~\eqref{eq: jt A final} is equal to 
    \[
        \V(u_1, \dotsc, u_p)
        \V(v_1, \dotsc, v_q)
        \dfrac
        {
            \prod_{i = 1}^{p} \prod_{j = 1}^{q} (1 - u_i v_j)
        }
        {
            \prod_{k = 1}^{n} \prod_{i = 1}^{p} (1 - x_k u_i) 
            \cdot 
            \prod_{k = 1}^{n} \prod_{j = 1}^{q} (x_k - v_j)
        },
    \]
    as we expected. This concludes the proof of the theorem.
\end{proof}


\section{N\"{a}gelsbach--Kostka identities}

In this section we introduce another ninth variation of the orthogonal and symplectic characters. We start from a Jacobi--Trudi identity, as originally done by Macdonald~\cite{Macdonald}. In this generality we derive orthogonal and symplectic N\"{a}gelsbach--Kostka identities.

\subsection{Original ninth variation}

First, we recall the original Macdonald's ninth variation of Schur functions.
Let $h_{r s}$ be independent indeterminates with $r \ge 1$ and $s \in \mathbb{Z}$. Set also $h_{0 s} = 1$ and $h_{r s} = 0$ for $r < 0$ and all $s \in \mathbb{Z}$. Consider the ring generated by the $h_{r s}$ and define an automorphism $\phi$ of this ring by $\phi (h_{r s}) = h_{r, s + 1}$ for all $r, s \in \mathbb{Z}$. Thus we can write $h_{r s} = \phi^s h_r$, where $h_r = h_{r 0}$.

For any two partitions $\lambda, \mu$ of lengths $\le n$ define
\[
    s_{\lambda / \mu} 
    = 
    \det \left[ \phi^{\mu_j + 1 - j} h_{\lambda_i -\mu_j - i + j} \right]_{i,j = 1}^n.
\]
For $\mu = \varnothing$ we have 
\[
    s_\lambda = \det \left[ \phi^{1 - j} h_{\lambda_i - i + j} \right]_{i,j = 1}^n.
\]
The same argument as in Macdonald~\cite[p.~71]{Macdonald-book} shows that $s_{\lambda / \mu} = 0$ unless $\mu \subset \lambda$, and $s_{\lambda / \mu} = 1$ when $\mu = \lambda$, due to the conditions imposed on the $h_{rs}$, $r \le 0$.

Specialising $h_{rs} = h_r \factorial{x}{\tau^s c}{}$ gives us the factorial Schur functions, as follows from the unflagged factorial Jacobi--Trudi identity~\eqref{eq: unflagged factorial identity}.
Setting $h_{rs} = s_{(r, 0, \dotsc, 0)}^\mathcal{F} (x_1, \dotsc, x_{n + s})$, we obtain the generalised Schur functions, due to the flagged Jacobi--Trudi identity~\eqref{eq: flagged identity}.
In particular, the $h_{rs} = h_r \factorial{x_1, \dotsc, x_{n + s}}{c}{}$ specialisation leads to the same factorial Schur function as $h_{rs} = h_r \factorial{x}{\tau^s c}{}$.

We necessarily have that 
\[
    s_{(r, 0, \dotsc, 0)} = h_r,
\]
and we define 
\[
    e_r = s_{(1^r)}
\]
for $r \ge 0$ and $e_r = 0$ for $r < 0$.

Macdonald proves the following N\"{a}gelsbach--Kostka identity:
\begin{equation}
    s_{\lambda / \mu} 
    = 
    \det \left[ \phi^{-\mu'_j - 1 + j} e_{\lambda'_i - \mu'_j - i + j} \right]_{i,j = 1}^{m},
    \label{eq: original nagelsbach-kostka}
\end{equation}
where $\lambda'$ and $\mu'$ are the conjugate partitions, and $\ell(\lambda'), \ell(\mu') \le m$.
It can be equivalently rephrased as follows. Consider the lower unitriangular matrices 
\begin{align*}
    & A = \left[ \phi^{1 - n + j} h_{i - j} \right]_{i, j = 0}^N, \\
    & B = \left[ (-1)^{i - j} \phi^{i - n} e_{i - j} \right]_{i, j = 0}^N,
\end{align*}
where $N + 1 = n + m$.
It is easy to see that the $n \times n$ minor of $A$ corresponding to the rows $\lambda_i + n - i$ and columns $\mu_j + n - j$, with $1 \le i, j \le n$, is equal to $s_{\lambda / \mu}$. The corresponding complementary cofactor of $B^t$ has row indices $n - 1 + i - \lambda'_i$ and column indices $n - 1 + j - \mu'_j$, with $1 \le i, j \le m$, and is equal to the right-hand side of~\eqref{eq: original nagelsbach-kostka}.

The N\"{a}gelsbach--Kostka identity is equivalent to the fact that the matrices $A$ and $B$ are inverses of each other. 

\begin{remark}
    The matrices $A$ and $B$ are not the transposes of the matrices $H$ and $E$ from~\cite[(9.3)]{Macdonald}, but they can be transformed into them using Macdonald's involution $\varepsilon$, see~\cite[(9.5)]{Macdonald}, as 
    \begin{align*}
        & A = \varepsilon \phi^{n - 1} (H^t), \\
        & B = \varepsilon \phi^{n - 1} (E^t).
    \end{align*}
\end{remark}

\subsection{Symplectic ninth variation}

Now we can define a more general ninth variation of the symplectic characters:
\begin{equation} \label{Def: 9th sp}
    \sp_\lambda 
    = 
    \dfrac{1}{2} 
    \det \left[ \phi^{1 - j} h_{\lambda_i - i + j} + \phi^{j - 1} h_{\lambda_i - i - j + 2} \right]_{i, j = 1}^n.
\end{equation}
Specialising $h_{rs} = h_r \factorial{x, \overline{x}}{\tau^{s - n}c }{}$ and setting $c_n = 0$ for $n < 0$, we obtain the factorial characters $\sp_\lambda \factorial{x}{c}{}$.

We note that in fact the entries in the first column are equal to $2 h_{\lambda_i - i + 1}$, which explains the appearance of the $\frac{1}{2}$ in front of the determinant. Hence we can restate the definition as follows. Consider the lower unitriangular matrix $A^+ = \left[ A^+_{ij} \right]_{i,j = 0}^N$ with entries 
\[
    A^+_{ij}
    =
    \begin{cases}
        A_{ij} + A_{i, \, 2(n - 1) - j} & \mbox{\text{if} $j < n - 1$}; \\
        A_{ij} & \mbox{\text{if} $j \ge n - 1$},
    \end{cases}
\]
where we assume that $A_{ij} = 0$ whenever $i$ or $j$ is not between $0$ and $N$.
Then, as above, the $n \times n$ minor of $A^+$ corresponding to the rows $\lambda_i + n - i$ and columns $n - j$ ($1 \le i, j \le n$) is equal to $\sp_{\lambda}$.

More generally, given another partition $\mu$ of length $\ell(\mu) \le n$, we can define the skew ninth variation character $\sp_{\lambda / \mu}$ as the minor of the matrix $A^+$ corresponding to the rows $\lambda_i + n - i$ and columns $\mu_j + n - j$ ($1 \le i, j \le n$).

Consider the lower unitriangular matrix $B^- = \left[ B^-_{ij} \right]_{i,j = 0}^N$ with entries 
\[
    B^-_{ij}
    =
    \begin{cases}
        B_{ij} - B_{2(n - 1) - i, j} & \mbox{\text{if} $i > n - 1$}; \\
        B_{ij} & \mbox{\text{if} $i \le n - 1$},
    \end{cases}
\]
with the same assumption.
The next lemma can be found in Fulton--Harris~\cite[Lemma~A.43]{FultonHarris}. 
We give a proof for the sake of completeness.
\begin{proposition}
    The matrices $A^+$ and $B^-$ are inverses of each other.
\end{proposition}
\begin{proof}
    Let us expand the $ij$-th entry of $A^+ B^-$:
    \[
        \begin{gathered}
        (A^+ B^-)_{ij}
        =
        \sum_{0 \le k < n - 1} (A_{ik} + A_{i, 2(n - 1) - k}) B_{kj} 
        + A_{i, n - 1} B_{n - 1, j} 
        + \\ + 
        \sum_{n - 1 < k \le N} A_{ik} (B_{kj} - B_{2(n - 1) - k, j}) 
        = 
        \sum_{k = 0}^N A_{ik} B_{kj}
        + 
        \sum_{0 \le k < n - 1}  A_{i, 2(n - 1) - k} B_{kj} \, 
        - \\ -
        \sum_{n - 1 < k \le N} A_{ik} B_{2(n - 1) - k, j}.
        \end{gathered}
    \]
    The first sum, $ \sum_{k = 0}^N A_{ik} B_{kj} $, is equal to $\delta_{ij}$ since the matrices $A$ and $B$ are the inverses of each other. Noticing that for $ B_{2(n - 1) - k, j}$ to be nonzero it is necessary that $k \le 2(n - 1)$, we can rewrite the last sum:
    \[
        \sum_{n - 1 < k \le N} A_{ik} B_{2(n - 1) - k, j} 
        = 
        \sum_{n - 1 < k \le 2 (n - 1)} A_{ik} B_{2(n - 1) - k, j}.
    \]
    Now it is left to note that 
    \[
        \sum_{0 \le k < n - 1}  A_{i, 2(n - 1) - k} B_{kj} 
        = 
        \sum_{n - 1 < k \le 2 (n - 1)} A_{ik} B_{2(n - 1) - k, j},
    \]
    which gives us 
    \[
        (A^+ B^-)_{ij} = \delta_{ij}.
    \]
    This completes the proof.
\end{proof}

It follows that each minor of $A^+$ is equal to the complementary cofactor of the transpose of $B^-$. This readily implies the following symplectic N\"{a}gelsbach--Kostka identity.
\begin{corollary}
    We have 
    \begin{equation}
        \sp_\lambda 
        = 
        \det \left[ \phi^{j - 1} e_{\lambda'_i - i + j} - \phi^{j - 1} e_{\lambda'_i - i - j} \right]_{i,j = 1}^{m}.
        \label{eq: ninth sp nagelsbach-kostka}
    \end{equation}
\end{corollary}

In particular, for the Foley--King factorial characters $\sp_\lambda \factorial{x}{c}{}$, we get 
\[
    \sp_\lambda \factorial{x}{c}{}
    = 
    \det \left[ e_{\lambda'_i - i + j} \factorial{x, \overline{x}}{\tau^{(j - 1) - n} c}{} - e_{\lambda'_i - i - j} \factorial{x, \overline{x}}{\tau^{(j - 1) - n} c}{} \right]_{i,j = 1}^{m}.
\]

\subsection{Orthogonal ninth variation}

In a similar way we can define a ninth variation of the orthogonal characters:
\begin{equation} \label{Def: 9th o}
    \o_\lambda
    = 
    \det 
    \left[ 
        \phi^{1 - j} h_{\lambda_i - i + j}
        - 
        \phi^{1 + j} h_{\lambda_i - i - j}
    \right]_{i, j = 1}^n.
\end{equation}
Specialising $h_{rs} = h_r \factorial{x, \overline{x}, 1}{\tau^{s - n - 1} c}{}$  and assuming $c_n = 0$ for $n < 0$, we obtain $\so_\lambda \factorial{x}{c}{}$.
Setting $h_{rs} = h_r \factorial{x, \overline{x}}{\tau^{s - n} c}{}$ and this time assuming $c_n = 0$ for $n \le 0$, we get $\o_\lambda \factorial{x}{c}{}$.

Consider the lower unitriangular matrix $A^\circ = \left[ A^\circ_{ij} \right]_{i,j = 0}^N$ with entries 
\[
    A^\circ_{ij}
    =
    \begin{cases}
        A_{ij} - A_{i, \, 2n - j} & \mbox{\text{if} $j < n$}; \\
        A_{ij} & \mbox{\text{if} $j \ge n$},
    \end{cases}
\]
where we assume that $A_{ij} = 0$ whenever $i$ or $j$ is not between $0$ and $N$. The $n \times n$ minor of $A^\circ$ corresponding to the rows $\lambda_i + n - i$ and columns $n - j$ ($1 \le i, j \le n$) is equal to $\o_{\lambda}$.

Once again, we can also define the skew character $\o_{\lambda / \mu}$  as the minor of the matrix $A^\circ$ with rows $\lambda_i + n - i$ and columns $\mu_j + n - j$.

Consider also the lower unitriangular matrix $B^\times = \left[ B^\times_{ij} \right]_{i,j = 0}^N$ with entries 
\[
    B^\times_{ij}
    =
    \begin{cases}
        B_{ij} + B_{2n - i, j} & \mbox{\text{if} $i > n$}; \\
        B_{ij} & \mbox{\text{if} $i \le n$},
    \end{cases}
\]
with the same assumption as above.

\begin{proposition}
    The matrices $A^\circ$ and $B^\times$ are inverses of each other.
\end{proposition}

\begin{proof}
    Expanding the $ij$-th entry of $A^\circ B^\times$, we get
    \[
        \begin{gathered}
        (A^\circ B^\times)_{ij} = \sum_{k = 0}^{n - 1} (A_{ik} - A_{i, 2n - k}) B_{kj} + A_{in} B_{nj} + \sum_{k = n + 1}^N A_{ik} (B_{kj} + B_{2n - k, j}) = \\
        = 
        \sum_{k = 0}^N A_{ik} B_{kj} - \sum_{k = 0}^{n - 1} A_{i, 2n - k} B_{kj} + \sum_{k = n + 1}^N A_{ik} B_{2n - k, j}.
        \end{gathered}
    \]
    The latter sum is equal to $\sum_{k = n + 1}^N A_{ik} B_{2n - k, j} = \sum_{k = n + 1}^{2n} A_{ik} B_{2n - k, j} = \sum_{k = 0}^{n - 1} A_{i, 2n - k} B_{kj}$, since $B_{2n - k, j}$ is zero for $k > 2n$, which gives us 
    \[
        (A^\circ B^\times)_{ij} = \sum_{k = 0}^N A_{ik} B_{kj} = \delta_{ij},
    \]
    as required.
\end{proof}

It follows that we have the following orthogonal N\"{a}gelsbach--Kostka identity.
\begin{corollary}
    We have 
    \begin{equation}
        \o_\lambda 
        =
        \dfrac{1}{2} 
        \det \left[ \phi^{i - 1} e_{\lambda'_i - i + j} + \phi^{1 - i} e_{\lambda'_i - i - j + 2} \right]_{i, j = 1}^m.
        \label{eq: ninth o nagelsbach-kostka}
    \end{equation}
\end{corollary}

In the factorial specialisations, we get
\[
    \so_\lambda \factorial{x}{c}{} = 
    \dfrac{1}{2} 
    \det \left[ e_{\lambda'_i - i + j} \factorial{x, \overline{x}, 1}{\tau^{(i - 1) - n - 1} c}{} + e_{\lambda'_i - i - j + 2} \factorial{x, \overline{x}, 1}{\tau^{(1 - i) - n - 1} c}{} \right]_{i, j = 1}^m.
\]
and
\[
    \o_\lambda \factorial{x}{c}{} =
    \det \left[ e_{\lambda'_i - i + j} \factorial{x, \overline{x}}{\tau^{(i - 1) - n} c}{} + e_{\lambda'_i - i - j + 2} \factorial{x, \overline{x}}{\tau^{(1 - i) - n} c}{} \right]_{i, j = 1}^m.
\]

\begin{remark}
    Note that the Littlewood duality~\cite[Theorem 2.3.2]{KoikeTerada} is lost. That is, the involution $\omega \colon \phi^s h_r \mapsto \phi^{-s} e_r$, which maps $s_\lambda$ to $s_{\lambda'}$ as follows from the N\"{a}gelsbach--Kostka identity~\eqref{eq: original nagelsbach-kostka}, does not map $\sp_\lambda$ to $\o_{\lambda'}$.
\end{remark}

\printbibliography

@article{shifted,
  title={Shifted Schur Functions},
  author={Okounkov, Andrei and Olshanski, Grigori},
  journal={St. Petersburg Math. J.},
  volume={9},
  number={2},
  pages={239-300},
  year={1998},
}

@article {Molev,
    AUTHOR = {Molev, A. I.},
     TITLE = {Comultiplication rules for the double {S}chur functions and
              {C}auchy identities},
   JOURNAL = {Electron. J. Combin.},
  FJOURNAL = {Electronic Journal of Combinatorics},
    VOLUME = {16},
      YEAR = {2009},
    NUMBER = {1},
     PAGES = {Research Paper 13, 44 pp},
   MRCLASS = {05E05},
  MRNUMBER = {2475536}
}

@article {Olsh_2019,
    AUTHOR = {Olshanski, Grigori},
     TITLE = {Interpolation {M}acdonald polynomials and {C}auchy-type
              identities},
   JOURNAL = {J. Combin. Theory Ser. A},
  FJOURNAL = {Journal of Combinatorial Theory. Series A},
    VOLUME = {162},
      YEAR = {2019},
     PAGES = {65--117}
}

@article {Okada,
    AUTHOR = {Okada, Soichi},
     TITLE = {A generalization of {S}chur's {$P$}- and {$Q$}-functions},
   JOURNAL = {S\'{e}m. Lothar. Combin.},
  FJOURNAL = {S\'{e}minaire Lotharingien de Combinatoire},
    VOLUME = {81},
      YEAR = {2020},
     PAGES = {Art. B81k, 50 pp},
   MRCLASS = {05E05},
  MRNUMBER = {4120149}
}

@incollection{4authors,
    AUTHOR = {Nakagawa, Jun and Noumi, Masatoshi and Shirakawa, Miki and
              Yamada, Yasuhiko},
     TITLE = {Tableau representation for {M}acdonald's ninth variation of
              {S}chur functions},
 BOOKTITLE = {Physics and combinatorics, 2000 ({N}agoya)},
     PAGES = {180--195},
 PUBLISHER = {World Sci. Publ., River Edge, NJ},
      YEAR = {2001}
}

@article {2018a,
    AUTHOR = {Foley, Ang\`ele M. and King, Ronald C.},
     TITLE = {Factorial characters of the classical {L}ie groups},
   JOURNAL = {European J. Combin.},
  FJOURNAL = {European Journal of Combinatorics},
    VOLUME = {70},
      YEAR = {2018},
     PAGES = {325--353}
}

@article {Foley-King-ninth,
    AUTHOR = {Foley, Ang\`ele M. and King, Ronald C.},
     TITLE = {Determinantal and {P}faffian identities for ninth variation
              skew {S}chur functions and {$Q$}-functions},
   JOURNAL = {European J. Combin.},
  FJOURNAL = {European Journal of Combinatorics},
    VOLUME = {93},
      YEAR = {2021},
     PAGES = {Paper No. 103271, 31 pp}
}

@incollection {Macdonald,
    AUTHOR = {Macdonald, I. G.},
     TITLE = {Schur functions: theme and variations},
 BOOKTITLE = {S\'{e}minaire {L}otharingien de {C}ombinatoire
              ({S}aint-{N}abor, 1992)},
    SERIES = {Publ. Inst. Rech. Math. Av.},
    VOLUME = {498},
     PAGES = {5--39},
 PUBLISHER = {Univ. Louis Pasteur, Strasbourg},
      YEAR = {1992}
}

@article {Szego,
    AUTHOR = {Olshanski, Grigori},
     TITLE = {Cauchy-{S}zeg\H{o} kernels for {H}ardy spaces on simple {L}ie
              groups},
   JOURNAL = {J. Lie Theory},
  FJOURNAL = {Journal of Lie Theory},
    VOLUME = {5},
      YEAR = {1995},
    NUMBER = {2},
     PAGES = {241--273}
}

@article {KoikeTerada,
    AUTHOR = {Koike, Kazuhiko and Terada, Itaru},
     TITLE = {Young-diagrammatic methods for the representation theory of
              the classical groups of type {$B_n,\;C_n,\;D_n$}},
   JOURNAL = {J. Algebra},
  FJOURNAL = {Journal of Algebra},
    VOLUME = {107},
      YEAR = {1987},
    NUMBER = {2},
     PAGES = {466--511}
}

@incollection {Sundaram,
    AUTHOR = {Sundaram, Sheila},
     TITLE = {Tableaux in the representation theory of the classical {L}ie
              groups},
 BOOKTITLE = {Invariant theory and tableaux ({M}inneapolis, {MN}, 1988)},
    SERIES = {IMA Vol. Math. Appl.},
    VOLUME = {19},
     PAGES = {191--225},
 PUBLISHER = {Springer, New York},
      YEAR = {1990}
}

@article {SergeevVeselov,
    AUTHOR = {Sergeev, A. N. and Veselov, A. P.},
     TITLE = {Jacobi-{T}rudy formula for generalized {S}chur polynomials},
   JOURNAL = {Mosc. Math. J.},
  FJOURNAL = {Moscow Mathematical Journal},
    VOLUME = {14},
      YEAR = {2014},
    NUMBER = {1},
     PAGES = {161--168}
}

@article {HoweRemarks,
    AUTHOR = {Howe, Roger},
     TITLE = {Remarks on classical invariant theory},
   JOURNAL = {Trans. Amer. Math. Soc.},
  FJOURNAL = {Transactions of the American Mathematical Society},
    VOLUME = {313},
      YEAR = {1989},
    NUMBER = {2},
     PAGES = {539--570}
}

@article {Koike,
    AUTHOR = {Koike, Kazuhiko},
     TITLE = {On the decomposition of tensor products of the representations
              of the classical groups: by means of the universal characters},
   JOURNAL = {Adv. Math.},
  FJOURNAL = {Advances in Mathematics},
    VOLUME = {74},
      YEAR = {1989},
    NUMBER = {1},
     PAGES = {57--86}
}

@article {borodin2011boundary,
    AUTHOR = {Borodin, Alexei and Olshanski, Grigori},
     TITLE = {The boundary of the {G}elfand-{T}setlin graph: a new approach},
   JOURNAL = {Adv. Math.},
  FJOURNAL = {Advances in Mathematics},
    VOLUME = {230},
      YEAR = {2012},
    NUMBER = {4-6},
     PAGES = {1738--1779}
}

@article {aggarwal2021free,
    AUTHOR = {Aggarwal, Amol and Borodin, Alexei and Petrov, Leonid and
              Wheeler, Michael},
     TITLE = {Free fermion six vertex model: symmetric functions and random
              domino tilings},
   JOURNAL = {Selecta Math. (N.S.)},
  FJOURNAL = {Selecta Mathematica. New Series},
    VOLUME = {29},
      YEAR = {2023},
    NUMBER = {3},
     PAGES = {Paper No. 36, 138}
}

@book {Stanley,
    AUTHOR = {Stanley, Richard P.},
     TITLE = {Enumerative combinatorics. {V}ol. 2},
 PUBLISHER = {Cambridge University Press, Cambridge},
      YEAR = {1999}
}

@book {Weyl,
    AUTHOR = {Weyl, Hermann},
     TITLE = {The {C}lassical {G}roups. {T}heir {I}nvariants and
              {R}epresentations},
 PUBLISHER = {Princeton University Press, Princeton, N.J.},
      YEAR = {1939},
     PAGES = {xii+302},
   MRCLASS = {20.0X},
  MRNUMBER = {0000255},
MRREVIEWER = {C. Chevalley},
}

@book{FultonHarris,
    AUTHOR = {Fulton, William and Harris, Joe},
     TITLE = {Representation Theory: A First Course},
    SERIES = {Graduate Texts in Mathematics},
    VOLUME = {129},
 PUBLISHER = {Springer-Verlag, New York},
      YEAR = {1991}
}

@article {Itoh,
    AUTHOR = {Itoh, Minoru},
     TITLE = {Schur type functions associated with polynomial sequences of
              binomial type},
   JOURNAL = {Selecta Math. (N.S.)},
  FJOURNAL = {Selecta Mathematica. New Series},
    VOLUME = {14},
      YEAR = {2009},
    NUMBER = {2},
     PAGES = {247--274}
}

@article {2018b,
    AUTHOR = {Foley, Ang\`ele M. and King, Ronald C.},
     TITLE = {Factorial {$Q$}-functions and {T}okuyama identities for
              classical {L}ie groups},
   JOURNAL = {European J. Combin.},
  FJOURNAL = {European Journal of Combinatorics},
    VOLUME = {73},
      YEAR = {2018},
 SORTTITLE = "b",
     PAGES = {89--113}
}

@article {JingRozhkovskaya,
    AUTHOR = {Jing, Naihuan and Rozhkovskaya, Natasha},
     TITLE = {Vertex operators arising from {J}acobi-{T}rudi identities},
   JOURNAL = {Comm. Math. Phys.},
  FJOURNAL = {Communications in Mathematical Physics},
    VOLUME = {346},
      YEAR = {2016},
    NUMBER = {2},
     PAGES = {679--701}
}

@article {Benkart,
    AUTHOR = {Benkart, Georgia and Shader, Chanyoung Lee and Ram, Arun},
     TITLE = {Tensor product representations for orthosymplectic {L}ie
              superalgebras},
   JOURNAL = {J. Pure Appl. Algebra},
  FJOURNAL = {Journal of Pure and Applied Algebra},
    VOLUME = {130},
      YEAR = {1998},
    NUMBER = {1},
     PAGES = {1--48},
}

@article {Stokke1,
    AUTHOR = {Stokke, Anna and Visentin, Terry},
     TITLE = {Lattice path constructions for orthosymplectic determinantal
              formulas},
   JOURNAL = {European J. Combin.},
  FJOURNAL = {European Journal of Combinatorics},
    VOLUME = {58},
      YEAR = {2016},
     PAGES = {38--51},
}

@article {Stokke2,
    AUTHOR = {Stokke, Anna},
     TITLE = {An orthosymplectic {P}ieri rule},
   JOURNAL = {Electron. J. Combin.},
  FJOURNAL = {Electronic Journal of Combinatorics},
    VOLUME = {25},
      YEAR = {2018},
    NUMBER = {3},
     PAGES = {Paper No. 3.37, 17},
}

@article {Stokke3,
    AUTHOR = {Patel, Aalekh and Patel, Harsh and Stokke, Anna},
     TITLE = {Orthosymplectic {C}auchy identities},
   JOURNAL = {Ann. Comb.},
  FJOURNAL = {Annals of Combinatorics},
    VOLUME = {26},
      YEAR = {2022},
    NUMBER = {2},
     PAGES = {309--327}
}

@article {Zeta,
    AUTHOR = {Nakasuji, Maki and Phuksuwan, Ouamporn and Yamasaki,
              Yoshinori},
     TITLE = {On {S}chur multiple zeta functions: a combinatoric
              generalization of multiple zeta functions},
   JOURNAL = {Adv. Math.},
  FJOURNAL = {Advances in Mathematics},
    VOLUME = {333},
      YEAR = {2018},
     PAGES = {570--619}
}

@article {Cummins_King,
    AUTHOR = {Cummins, C. J. and King, R. C.},
     TITLE = {Composite {Y}oung diagrams, supercharacters of {${\rm
              U}(M/N)$} and modification rules},
   JOURNAL = {J. Phys. A},
  FJOURNAL = {Journal of Physics. A. Mathematical and General},
    VOLUME = {20},
      YEAR = {1987},
    NUMBER = {11},
     PAGES = {3121--3133}
}

@article {Balantekin_Bars,
    AUTHOR = {Balantekin, A. Baha and Bars, Itzhak},
     TITLE = {Representations of supergroups},
   JOURNAL = {J. Math. Phys.},
  FJOURNAL = {Journal of Mathematical Physics},
    VOLUME = {22},
      YEAR = {1981},
    NUMBER = {8},
     PAGES = {1810--1818}
}

@article {Littlewood_article,
    AUTHOR = {Littlewood, D. E.},
     TITLE = {On orthogonal and symplectic group characters},
   JOURNAL = {J. London. Math. Soc.},
  FJOURNAL = {Journal of the London Mathematical Society. Second Series},
    VOLUME = {30},
      YEAR = {1955},
     PAGES = {121--122}
}

@book {Littlewood_book,
    AUTHOR = {Littlewood, Dudley E.},
     TITLE = {The theory of group characters and matrix representations of
              groups},
      NOTE = {Reprint of the second (1950) edition},
 PUBLISHER = {AMS Chelsea Publishing, Providence, RI},
      YEAR = {2006},
     PAGES = {viii+314}
}

@incollection {Olsh_Formalism,
    AUTHOR = {Ol'shanskii, G. I.},
     TITLE = {Unitary representations of infinite-dimensional pairs
              {$(G,K)$} and the formalism of {R}. {H}owe},
 BOOKTITLE = {Representation of {L}ie groups and related topics},
    SERIES = {Adv. Stud. Contemp. Math.},
    VOLUME = {7},
     PAGES = {269--463},
 PUBLISHER = {Gordon and Breach, New York},
      YEAR = {1990}
}

@article {Foley-King-super,
    AUTHOR = {Foley, Ang\`ele M. and King, Ronald C.},
     TITLE = {Factorial supersymmetric skew {S}chur functions and ninth
              variation determinantal identities},
   JOURNAL = {Ann. Comb.},
  FJOURNAL = {Annals of Combinatorics},
    VOLUME = {25},
      YEAR = {2021},
    NUMBER = {2},
     PAGES = {229--253}
}

@book {Macdonald-book,
    AUTHOR = {Macdonald, I. G.},
     TITLE = {Symmetric functions and {H}all polynomials},
    SERIES = {Oxford Mathematical Monographs},
   EDITION = {Second},
      NOTE = {With contributions by A. Zelevinsky,
              Oxford Science Publications},
 PUBLISHER = {The Clarendon Press, Oxford University Press, New York},
      YEAR = {1995},
     PAGES = {x+475}
}

@article {Rains,
    AUTHOR = {Rains, Eric M.},
     TITLE = {{${\rm BC}_n$}-symmetric polynomials},
   JOURNAL = {Transform. Groups},
  FJOURNAL = {Transformation Groups},
    VOLUME = {10},
      YEAR = {2005},
    NUMBER = {1},
     PAGES = {63--132}
}

\end{document}